\documentclass[a4paper]{amsart}

\usepackage{latexsym,amscd,amssymb,amsmath,amsthm,comment,url,color,mathrsfs}
\usepackage{graphicx,xcolor}
\usepackage{pb-diagram,tikz}
\usepackage[all]{xy}

\title{Lyubeznik numbers of the face rings of simplicial posets}

\author{Kosuke Shibata}
\address{National Institute of Technology, Yonago College, Yonago, Tottori, 683-8502, JAPAN}
\email{shibata@yonago.kosen-ac.jp}

\author{Kohji Yanagawa} 
\address{Department of Mathematics, 
Kansai University, Osaka, 564-8680, Japan.}
\email{yanagawa@kansai-u.ac.jp}

\theoremstyle{plain}
  \newtheorem{thm}{Theorem}[section]
  \newtheorem{prop}[thm]{Proposition}
  \newtheorem{lem}[thm]{Lemma}
  \newtheorem{cor}[thm]{Corollary}
  
\theoremstyle{definition}
  \newtheorem{dfn}[thm]{Definition}
  \newtheorem{exmp}[thm]{Example}
\theoremstyle{remark}
  \newtheorem{rem}[thm]{Remark}
\let\opn\operatorname 

\def\@bothmode#1{\ifmmode #1\else $#1$\fi}
\def\@autopr#1{%
   \@chCount{#1}%
   \ifnum\@tempchn<2 #1\else (#1)\fi
}
\numberwithin{equation}{section}
\def\NN{\mathbb{N}} 
\def\ZZ{\mathbb{Z}} 

\def\MM{\mathbb{M}}
\def\C{\check{C}}
\def\supp{\opn{supp}} 
\def\bc{a} 
\def\chara{\operatorname{char}} 

\def\defopn#1{%
    \@xp\def\csname #1\endcsname{%
        \def\@tempopn{\opn{\csname the#1\endcsname}}%
        \@ifnextchar_{\@opform}{\@opform_{}}%
    }%
}
\def\theE{E}
\def\E{\@ifstar{{}^*\theE}{\theE}} 
\def\gmod{{}^*\!\mod}
\def\gModN{\*Mod_{\ge \b0}}
\def\*Mod{^*\!\Mod}
\def\Mod{\operatorname{Mod}}
\def\mod{\operatorname{mod}}

\def\<{\langle}
\def\>{\rangle}
\def\theD{D}

\def\lexico{\succ_{\text{lex}}}
\def\depth{\operatorname{depth}}
\def\rank{\operatorname{rank}}

\def\chara{\operatorname{char}}

\def\cpx#1{#1^{\bullet}} 
\def\theD{D} 
 
\def\Sh{\operatorname{Sh}} 
\def\sp{\operatorname{Sp\acute{e}}} 
\def\cDx{\cpx {\mathcal D}_X} 
\def\const{\underline{K} } 
\def\mdL{\operatorname{mod} \Lambda}
\def\<{\langle}
\def\>{\rangle}

\def\DC{D^\bullet}

\def\fm{\mathfrak{m}}
\def\fp{\mathfrak{p}}

\def\b0{\mathbf 0}
\def\ba{\mathbf a}
\def\bb{\mathbf b}
\def\bc{\mathbf c}
\def\bd{\mathbf d}
\def\be{\mathbf e}
\def\b0{\mathbf 0}
\def\bt{\mathbf t}
\def\bw{\mathbf w}

\def\11{\mathbf 1}
\def\ua{{\underline{\ba}}}
\def\ub{{\underline{\bb}}}

\def\sm{{\sf m}}

\def\EE{{}^*\! E}

\def\too{\longrightarrow}
\def\wt{\widetilde{t}}

\usepackage{bm}
\def\Sq{\operatorname{Sq}}
\def\Str{\operatorname{Str}}

\def\Ext{\operatorname{Ext}}
\def\Hom{\operatorname{Hom}}
\def\gHom{\operatorname{*\!\Hom}}
\def\gExt{\operatorname{*\!\Ext}}

\def\Nf{\mathcal{N}}
\def\wt{{\widetilde{t}}}

\def\fp{{\mathfrak{p}}}

\def\C{\check{\mathcal C}}
\def\Lf{\mathcal{L}}
\def\Nf{\mathcal{N}}
\def\DC{D^\bullet}
\def\DB{{}^*\! \DC}
\def\IB{{}^*\! I^\bullet_S}

\def\bt{\mathbf t}
\def\too{\longrightarrow}

\def\Im{\operatorname{Im}}
\def\Ker{\operatorname{Ker}}
\def\Coker{\operatorname{Coker}}

\def\injdim{\operatorname{inj\, dim}}

\def\SExt{\mathcal{E}xt}

\def\lm{\operatorname{LM}}
\def\lv{\operatorname{LV}}
\def\nv{\operatorname{NV}}
\def\id{\operatorname{id}}

\def\nf{\operatorname{nf}}

\def\ol{\overline}
\makeatother

\begin{document}

\begin{abstract}
For a {\it simplicial poset} $P$, Stanley assigned the face ring $A_P$, which is the quotient of the polynomial ring $S:=K[t_x \mid x \in P \setminus \{\widehat{0} \}]$ by the ideal $I_P$. This is a generalization of Stanley-Reisner rings, but the grading of $S$ is more involved.  
Building on our previous paper, we further develop and complete the study of the local cohomology modules $H_{I_P}^i(S)$. 
In our categorical framework, $\Ext_S(A_P, S)$ ``encodes" all the information of $H_{I_P}^i(S)$. 
In particular, Lyubeznik numbers $\lambda_{i,j}(A_P)$ can be simply expressed in terms of  $[\Ext_S^*(\Ext_S^*(A_P,S), S)]_\b0$, and depend only on the homeomorphism class of the underlying space of the CW complex associated with $P$ and $\chara(K)$.

\end{abstract}

\maketitle

\section{Introduction}
Let $(R, \fm)$ be a regular local ring of dimension $d$ containing a field, $I$ an ideal of $R$, 
and $A:=R/I$ the quotient ring. The local cohomology module $H_I^i(R)$ is very mysterious 
(it is neither noetherian nor artinian in general).
The {\it Lyubeznik number}  $\lambda_{i,j}(A)$ is defined as the $i$-th Bass number of $H_I^{d-j}(R)$, that is, 
$$\lambda_{i,j}(A):= \mu_i(\fm,  H_I^{d-j}(R))=\dim_{R/\fm} \Ext_R^i(R/\fm, H_I^{d-j}(R)).$$
These numbers are finite and depend only on the ring $A$, $i$ and $j$ (\cite{L93,L97}), and  they contain important topological data (see \cite{L93, GS, B07, BB05}).

For a polynomial ring $S=K[x_1, \ldots, x_n]$ over a field $K$ and a graded ideal $I$, $\lambda_{i,j}(S/I)$ is defined in the same way using the graded maximal ideal $\fm=(x_1, \ldots, x_n)$ and $K=S/\fm$.
The case of Stanley–Reisner ideals (i.e.,  squarefree monomial ideals) has been studied extensively. See \cite{Y01a,AV}.
In this paper, we extend this theory to the ideals associated with {\it simplicial posets}.

We say that a poset $P$ is simplicial, if it is a finite poset with the smallest element $\hat{0}$, and the subposet $\{ y \in P \mid y \leq x  \}$ is isomorphic to a boolean algebra for all $x \in P$. 
If we regard the set of faces of a simplicial complex as a poset by inclusions, it is simplicial. Moreover, any simplicial poset is isomorphic to the face poset of a certain regular cell complex. 

For $x_1, \ldots, x_m \in P$, $[x_1 \vee \cdots \vee x_m ]$ denotes the set of minimal elements of $\{ \, z \in P \mid z \geq x_1, \ldots, x_m \, \}$. 
If $[ x \vee y] \ne \emptyset$, then  
$\{ \, z \in P \mid z \leq x, y \, \}$ has the greatest element $x \wedge y$. 
Let us recall Stanley's construction  of the face ring of $P$. 
For $P^*:=P\backslash \{\hat{0}\}$,   
let  $S := K[ \, t_x \mid x \in P^* \, ]$ be the polynomial ring. Set 
\begin{equation}\label{f_xy}
f_{x,y}:=t_xt_y - t_{x \wedge y} \sum_{z \in [x \vee y]} t_z
\end{equation}
for $x,y \in P^*$. Here we set  $t_{\hat{0}}=1$ for convenience, and 
if $[x \vee y] = \emptyset$, we set $f_{x,y} =t_xt_y$.     
Now consider the ideal  
$$I_P:= (\, f_{x,y} | \ x,y \in P^*   \, ) \subset S$$
and call $A_P:= S/I_P$ the {\it face ring} of $P$.  See Example~\ref{running} below for a concrete illustration. If the simplicial poset $P$ comes from a simplicial complex $\Delta$, then $A_P$ is isomorphic to the Stanley--Reisner ring $K[\Delta]$ of $\Delta$. 
M. Masuda and his co-workers studied $A_P$ from the viewpoint of toric topology, since the equivariant cohomology ring of a torus manifold is closely related to $A_P$ (see \cite{M05,MP06}).

Let $y_1, \ldots, y_n$ be the rank 1 elements of $P$.  For simplicity, set $t_i:=t_{y_i}$. 
We equip $S$ with a $\ZZ^n$-grading by setting $\deg t_x =\sum_{y_i \le x} \be_i \in \NN^n$ (in particular, $\deg t_i=\be_i$), where $\be_i$ denotes the $i$-th coordinate vector. 
Then $I_P$ is a $\ZZ^n$-graded ideal, and $A_P$ is a $\ZZ^n$-graded ring. 
This grading induces the $\ZZ$-grading, which is {\it not} standard in most cases (i.e., $\deg_\ZZ t_x \ge 2$ for $x \in P^* \setminus \{y_1, \ldots, y_n\}$). 
Moreover, we often have $\dim_K S_\ba \ge 2$ for $\ba \in \NN^n$. 
 These points make our study considerably more involved. Let $\*Mod S$ be the category of $\ZZ^n$-graded $S$-modules, and $\gmod S$ its full subcategory consisting of finitely generated modules. 

 A $\ZZ^n$-graded prime ideal of $S$ containing $I_P$ is of the form $\fp_x:=(t_z \mid z \not \le x)+I_P$ for some $x \in P$. 
 Note that $S_x := S/\fp_x\cong  K[t_i \mid y_i \le x]$.  
In the earlier paper \cite{Y11}, the second author introduced {\it squarefree modules} over $A_P$ (here, we regard them as $S$-modules). For example, $S_x$ for $x \in P$ and $A_P$ are squarefree.  
Let $\Sq A_P$ be the full subcategory of $\gmod S$ consisting of squarefree modules. Then $\Sq A_P$ has enough injectives, and an indecomposable injective is isomorphic to $S_x$ for some $x \in P$.  Let $\omega_S$ be the $\ZZ^n$-graded canonical module of $S$, which is a degree shift of the free module $S$. Then $M \in \Sq A_P$ implies $\gExt_S^i(M, \omega_S) \in \Sq A_P$ (Lemma~\ref{Ext is sq}). 

 In the preceding paper \cite{SY}, we constructed  the 
 injective envelope $\EE_x$ of $S_x$ in $\*Mod S$. If $x \ge z$, there are far too many $\ZZ^n$-graded $S$-homomorphisms $\EE_x \to \EE_z$. In \cite{SY}, we defined the {\it clean} homomorphism  $\EE_x \to \EE_z$, which is unique up to scalar multiplication. 
 Roughly speaking, we call a module $M \in \*Mod S$ {\it straight} if it is equipped with an injective presentation
 $$p_M : 0 \too M \too E_M \stackrel{\psi}{\too} E'_M \quad  (\ E_M =  \bigoplus_{x \in P} \EE_x^{l_x}, \  E_M'= \bigoplus_{x \in P} \EE_x^{m_x})$$
 such that $\psi$ is a linear combination of clean maps. A $\ZZ^n$-graded homomorphism $\varphi: M \to N$ is clean if it is induced by a chain map between the presentations $p_M, p_N$ such that $E_M \to E_N$ and $E_M' \to E_N'$ are linear combinations of clean maps (two such diagrams are identified if the maps $M \to N$ coincide). 
Let $\Str_P S$ be the category consisting of straight modules and clean morphisms between them. 
Up to equivalence, we may regard $\Str_P S$ as  a {\it non-full} subcategory of $\*Mod S$. 
For $M \in \Str_P S$, the module 
$\Nf(M):=[(0:_M I_P)]_{\ge \b0}$ belongs to $\Sq A_P$, and $M$ is a $\ZZ^n$-graded essential extension of $\Nf(M)$ (for example, we have $\Nf(\EE_x) \cong S_x$). 
The functor $\Nf$ induces an equivalence $\Sq A_P \cong \Str_P  S$, and $\Str_P S$ is abelian (Theorem~\ref{sq-str}). Let ${}^*\!I_S^\bullet$ be the $\ZZ^n$-graded injective resolution of $\omega_S$, which is a translation of the $\ZZ^n$-graded dualizing complex of $S$. 
Then $\Gamma_{I_P}({}^*\!I_S^\bullet)$ is a cochain complex in $\Str_P S$, where $\Gamma_{I_P}$ is the $I_P$-torsion functor.  Hence the $i$-th cohomology of $\Gamma_{I_P}({}^*\!I_S^\bullet)$ is the local cohomology module  $H_{I_P}^i(\omega_S)$, which is isomorphic to $H_{I_P}^i(S)$ up to a degree shift.

Let $\varphi: M \to N$ be a morphism in $\Str_P S$.  We will show that the kernel, image, and cokernel of $\varphi$ computed in $\Str_P S$ coincide with those computed in $\*Mod S$ 
(Theorems~\ref{Ker is straight} and \ref{Coker is straight}). The proofs are quite involved, and this part is one of the technical cores 
of the present paper. 
This leads to the following results, which are generalizations of the main results in \cite{Y01a}. 

\medskip

\noindent{\bf Theorem~\ref{localcoh-ext}}
{\it The local cohomology module $H_{I_P}^i(\omega_S)$, which is a degree shift of $H_{I_P}^i(S)$,  is a straight module with $\Nf (H_{I_P}^i(\omega_S)) \cong \gExt^i_S(A_P,\omega_S)$.}

\medskip

\noindent{\bf Theorem~\ref{Lyubeznik number}}
{\it Set $p:=\dim S= \#P^*$. The Lyubeznik numbers of $A_P$ are given by
$$\lambda_{i,j}(A_P)=\dim_K [\gExt_S^{p-i}(\gExt_S^{p-j}(A_P,\omega_S),\omega_S)]_{\mathbf{0}}.$$}

\medskip

Recall that a simplicial poset $P$ gives a regular cell complex $\Gamma(P)$. In \cite{Y11}, the second author associated a constructible sheaf $M^+$ on $\Gamma(P)$ with each $M \in \Sq A_P$. 
In this context, the duality ${\mathrm R}\!  \gHom_S(-, \omega_S)$ on $D^b(\Sq A_P)$ corresponds to the Poincar\'e-Verdier duality on the constructible sheaves on $\Gamma(P)$ up to a cohomological shift. 
Now we have the following, which generalizes a main result of \cite{AY}. 

\medskip

\noindent{\bf Theorem~\ref{topological}}
{\it The Lyubeznik numbers $\lambda_{i,j}(A_P)$ depend only on the homeomorphism class of the underlying space of  $\Gamma(P)$ and $\chara(K)$.}

\medskip

\section{Preliminaries}
Throughout the paper, let $P$ be a simplicial poset. 
This section provides a concise summary of the main definitions, constructions, and results from our preceding paper \cite{SY}, which  will be heavily used in the present work.  

\begin{exmp}\label{running}
The poset $P_1$ given by the following Hasse diagram is simplicial. We have $[y_1 \vee y_2]=\{x,z\}$, $y_1 \wedge y_2= \hat{0}$ and $[x \vee z]=\emptyset$ ($x \wedge z$ does not exist). 
Recall the convention that $t_i$ denotes $t_{y_i} \in S$ for $i=1,2$. 
The polynomials defined in \eqref{f_xy} are 
$f_{y_1, y_2}=t_1t_2-(t_x+t_z)$ and $f_{x,z}=t_xt_z$, and the face ring $A_{P_1}$ is given by 
$$S/I_{P_1}=K[t_1, t_2, t_x, t_z]/(t_1t_2-t_x-t_z, t_xt_z)$$
(note that, for a general $P$, if $x$ and $y$ are comparable, then $f_{x,y}=0$.)

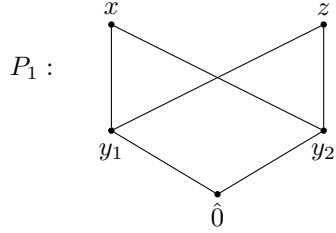
\begin{center}
\begin{figure}[ht]
\begin{tikzpicture}[scale=0.7, every node/.style={circle, draw, fill=black, inner sep=0.7pt}]

\node[label=above:$P_1:$,color=white] (a) at (-3.5,0.6) {};
\node[label=above:$x$] (x) at (-2,2) {};
\node[label=above:$z$] (z) at (2,2) {};
\node[label=below:$y_1$] (y_1) at (-2,0) {};
\node[label=below:$y_2$] (y_2) at (2,0) {};
\node[label=below:$\hat{0}$] (0) at (0,-1.2) {};

\draw (x) -- (y_1);
\draw (x) -- (y_2);
\draw (z) -- (y_1);
\draw (z) -- (y_2);
\draw (y_1) -- (0);
\draw (y_2) -- (0);
\end{tikzpicture}
\end{figure}
\end{center}

\end{exmp}

The $\ZZ^n$-graded prime ideals of $S$ containing $I_P$ are of the form 
$$\fp_x:=(t_z \mid z \not \le x)+I_P$$ for some $x \in P$. Note that $\fp_{\hat{0}}$ is the graded maximal ideal $\fm:=(t_z \mid z \in P^*)$. Set $$S_x:=S/\fp_x \cong K[t_i \mid y_i \le x]$$
(note that $S_{\hat{0}} =S/\fm \cong K$). 
For simplicity, we denote the image of $t_i$  in $S_x$ just by $t_i$, not something like $\overline{t_i}$. Let $\pi_x$ be the natural surjection $\pi_x : S\to S_x \, (=S/\fp_x)$. 

Let $\*Mod S$ be the category of $\ZZ^n$-graded $S$-modules and degree-preserving $S$-homomorphisms between them. 
For $M \in \*Mod S$ and $\ba \in \ZZ^n$, 
$M(\ba)$ denotes the $\ZZ^n$-graded $S$-module which coincides with $M$ as the underlying $S$-module
and the grading is given by $M(\ba)_{\bb}=M_{\ba +\bb}$ for $\bb \in \ZZ^n$. 
For $M,N \in \*Mod S$ and $\ba \in \ZZ^n$, we set 
$$[\gHom_S(M,N)]_\ba:=\Hom_{\*Mod S}(M,N(\ba))
$$ 
(in particular,  $[\gHom_S(M,N)]_\b0=\Hom_{\*Mod S}(M,N)$) and 
$$\gHom_S(M,N):=\bigoplus_{\ba \in \ZZ^n}[\gHom_S(M,N)]_\ba.$$ 
We can regard $\gHom_S(M,N) \in \*Mod S$ in the natural way.  
If $M$ is finitely generated, the underlying module of  $\gHom_S(M,N)$ is isomorphic to the usual $\Hom_S(M,N)$.  The same is true for $\gExt^i_S(M,N)$.

For $x \in P$, we will give an explicit description of the injective envelope $\EE_x$ of $S_x=S/\fp_x$ in $\*Mod S$. 
Clearly,  $\EE_{\widehat{0}}=E_S(K)$ is isomorphic to $K[t_z^{-1} \mid z \in P^*]$. 
So we may assume that $r:= \rank_P x \ge 1$ with $\{ i \mid y_i \le x \}=\{1, \ldots, r \}$ (equivalently, $x\in [y_1 \vee \cdots \vee y_r]$). Note that $S_x\cong K[t_1, \ldots, t_r]$ now. 
We also set $P^*_{-x}:=P^* \setminus \{ y_1, \ldots, y_r \}$. 
 
The $\ZZ^n$--graded $K$-vector space $\EE_x=K[t_1^{\pm 1}, \ldots, t_r^{\pm 1}] \otimes_K K[t_z^{-1} \, | \, z \in P^*_{-x}]$ admits the basis  
\begin{equation}\label{monomials in E_x} 
\left \{\alpha = \prod_{i=1}^r t_i^{a_i} \otimes \prod_{z \in P^*_{-x}} t_z^{-a_z} \, \middle |  \, (a_1, \ldots, a_r) \in \ZZ^r, (a_z) \in \NN^{P^*_{-x}} \right \} 
\end{equation}
We call the basis element $\alpha$  a {\it monomial} of $\EE_x$. We endow $\EE_x$ with a $\ZZ^n$--graded $S$-module structure as follows. First, for $w \in P^*$,  set 
$$t_w \cdot \prod_{i=1}^r t_i^{a_i} := 
\pi_x(t_w) \cdot  \prod_{i=1}^r t_i^{a_i},$$
where the multiplication on the right-hand side is computed in  the Laurent polynomial ring $K[t_1^{\pm 1}, \ldots, t_r^{\pm 1}]$. Next, we set 
$$t_i \cdot \prod_{z \in P^*_{-x}} t_z^{-a_z} =0 \qquad  (1 \le i \le r),$$
and 
$$t_w \cdot\prod_{z \in P^*_{-x}} t_z^{-a_z} := t_w^{-a_w+1}\prod_{w \ne z \in P^*_{-x}} t_z^{-a_z}$$
for $w \in P^*_{-x}$. 
Here, if $a_w=0$, then $t_w \cdot \prod_{z \in P^*_{-x}}t_z^{-a_z}=0$.  Finally, through the ring homomorphism $\Delta: S \to  S \otimes_K S$ defined
by $$\Delta(t_z) = t_z \otimes 1 + 1 \otimes t_z$$ for all $z \in P^*$, we regard $\EE_x$ as a $\ZZ^n$-graded $S$-module. 

Let $1_x \in \EE_x$ denote the monomial $1 \otimes 1$ in \eqref{monomials in E_x} (i.e., $a_i=a_z =0$ for all $1 \le i \le r$ and $z \in P^*_{-x}$). The submodule generated by $1_x$ is isomorphic to $S_x$, and we regard $S_x$ as a submodule of $\EE_x$ in this way. Then $\EE_x$ is a *essential extension of $S_x$ (i.e., any non-zero $\ZZ^n$-graded submodule $M \subset \EE_x$ satisfies $M \cap S_x \ne 0$). 

Clearly, $\Hom_S(S_x,S_z) \ne 0$ if and only if $x \ge z$.  
If $x \ge z$, for all $f \in [\gHom_S(S_x,S_z)]_\b0$, we have $f = c\pi$ for some $c \in K$, where $\pi :S_x \to S_z$ is the canonical surjection (note that $\fp_x \subset \fp_z$). Similarly, $\Hom_S(\EE_x,\EE_z) \ne 0$ if and only if $x \ge z$. For $\varphi \in [\gHom_S(\EE_x,\EE_z)]_\b0$, we have $\varphi(1_x)=c 1_z$ for some $c \in K$.  However, $\varphi(1_x)=0$ can occur for  $0 \ne \varphi \in [\gHom_S(\EE_x,\EE_z)]_\b0$. 

Throughout the rest of the paper, whenever we consider a direct sum of $\EE_x$ for various $x \in P$, we fix a direct sum decomposition and a basis (namely, the set of monomials) for each summand.

\begin{dfn}\label{clean map def}
For $x \in P$, consider  the $K$-subspace 
$$\Lf(\EE_x):=\<  \, u \otimes v \in \EE_x 
\mid \text{$u \otimes v$ is a monomial as in \eqref{monomials in E_x}, } 
v \ne 1, \deg(u \otimes v) \ge \b0 \, \>$$
of $\EE_x$.  
For $E=\bigoplus_{x \in P} \EE_x^{\, l_x} \in \*Mod S$ with $l_x \in \NN$, set 
$$\Lf(E):=\bigoplus_{x \in P} \Lf(\EE_x)^{l_x}.$$
\end{dfn}

Clearly, $\Lf(\EE_x)$ depends on the choice of basis (i.e., the set \eqref{monomials in E_x} of monomials) of $\EE_x$.  
It is easy to see that $[\EE_x]_{\ge \b0} = S_x \oplus \Lf(\EE_x)$ as a $\ZZ^n$-graded $K$-vector space.  

\begin{dfn}
For $\ZZ^n$-graded $S$-modules  $E=\bigoplus_{x \in P} \EE_x^{l_x}$ and $F =\bigoplus_{x \in P} \EE_x^{m_x}$ with $l_x, m_x \in \NN$, we say $\psi \in [\gHom_S(E,F)]_{\b0}$ is {\it clean}, if $\psi(\Lf(E)) \subset \Lf(F)$.  
\end{dfn}    

If $x \ge x'$, there is a non-zero clean morphism $\varphi: \EE_x\to \EE_{x'}$, which is unique up to scalar multiplication (\cite[Lemma~3.7 and Proposition~3.10]{SY}). For the reader's convenience, we recall the explicit description of a clean map $\varphi:\EE_x \to \EE_{x'}$ given in \cite[\S 3]{SY}. If $x=x'$, then $\varphi= c \operatorname{id}_{\EE_x}$ for some $c \in K$, so we may assume that $x>x'$. Then it suffices to consider the case $x$ covers $x'$. Indeed, the general case is obtained by composing maps of this type. 
We may assume that $x \in  [y_1 \vee \ldots \vee y_r]$ and $x' \in  [y_1 \vee \ldots \vee y_{r-1}]$. 

Set 
\begin{equation}\label{def of Z}
Z:= \{ z \in P \mid z \le x, z \not \le x', z \ne y_r \}
\end{equation}
and 
\begin{equation}\label{def of W}
W:= P^* \setminus (Z \cup \{y_1, \ldots, y_r\}).
\end{equation}
Consider (Laurent) monomials $\prod_{i=1}^r t_i^{a_i} \in K[t_1^{\pm 1}, \ldots, t_r^{\pm 1}]$, $\bt_Z^{-\bb} :=\prod_{z \in Z} t_z^{-b_z}$ for $\bb =(b_z)_{z \in Z} \in \NN^Z$ and $\bt_W^{-\bc} :=\prod_{w \in W} t_w^{-c_w}$ for $\bc=(c_w)_{w \in W} \in \NN^W$. Take 
\begin{equation}\label{alpha}
\alpha= \left (\prod_{i=1}^r t_i^{a_i} \right )
\otimes \bt_Z^{-\bb} \bt_W^{-\bc} \in \EE_x. 
\end{equation}
For a monomial $\bt_Z^{\bd} \in S$, we have $(e_1, \ldots, e_r) \in \ZZ^r$ with
$$
\pi_x(\bt_Z^{\bd})\cdot \prod_{i=1}^r t_i^{a_i}  =\prod_{i=1}^rt_i^{e_i}.$$
Then we set 
$$(\bt^{\bd}_Z | \alpha):=
\displaystyle \prod_{z \in Z} \binom{b_z+d_z}{b_z} \cdot\left( \prod_{i=1}^{r-1}t_i^{e_i}  \otimes \bt_Z^{-\bb-\bd} \cdot \bt_W^{-\bc} \cdot t_r^{e_r} \right) \in \EE_{x'}.
$$
Here we interpret $(\bt^{\bd}_Z | \alpha)=0$ if $e_r>0$. 
Now we define $\varphi: \EE_x \to \EE_{x'}$ by 
\begin{equation}\label{def of varphi}
\varphi(\alpha):=\sum_{\bd\in \NN^Z} (\bt^{\bd}_Z | \alpha).
\end{equation}
If $\sum_{z \in Z}d_z > -a_r$, then $(\bt^{\bd}_Z | \alpha)=0$. 
Hence, the right side of \eqref{def of varphi} is a finite sum.

\section{Squarefree $A_P$-modules revisited}
 For the reader's convenience, we review the notion of {\it squarefree modules} over $A_P$. We refer the reader to \cite{Y11} for further details. Set 
$$\MM(x):=\bigoplus_{y_i\leq x}\NN \be_i^{x},$$
where $\be_i^x$ is a basis element. So $\MM(x)\cong \NN^{\rank x}$ as additive monoids. For $x, z\in P$ with $x\leq z$, define an injection $\iota_{z,x}:\MM(x)\ni \be_i^x\mapsto \be_i^z\in \MM(z)$ of monoids. Set
$$\MM:=\varinjlim_{x\in P}\MM(x),$$
where the colimit is taken in the category of sets  with respect to the maps $\iota_{z,x}:\MM(x)\to \MM(z)$ for $x,z\in P$ with $x\leq z$. Note that $\MM$ need not be a monoid, and we can regard $\MM(x)$ as a subset of $\MM$. For each $\ua\in \MM$, $\{x\in P\mid \ua\in \MM(x)\}$ has the smallest element, which is denoted by $\sigma(\ua)$. 


We say that a monomial $\sm=\prod_{x\in P^*} t_x^{n_x}\in A_P$ ($n_x \in \NN$) is {\it standard}, if $\{x \in P \mid n_x \ne 0\}$ is a chain. In this case, set $\sigma(\sm):=\max\{x\in P\mid n_x\neq 0\}$. 
For convenience, we also regard $1 \in A_P$ as a standard monomial with $\sigma(1) =\hat{0}$. 
As shown in \cite{St91}, the set of standard monomials forms a $K$-basis of $A_P$. There is a one-to-one correspondence between the elements of $\MM$ and the standard monomials of $A_P$. We denote the standard monomial corresponding to $\ua \in \MM$ by $t^{\ua}$. Clearly, $\sigma(\ua)=\sigma(t^\ua)$. 

Let $\ua, \ub\in \MM$. For $x \in [\sigma(\ua)\vee \sigma(\ub)]\neq \emptyset$, we can regard $\ua, \ub \in \MM(x)$, and have the sum $\ua+\ub \in \MM(x)$, which is denoted by $(\ua+\ub)|_x$. 
Since each $\MM(x)$ is $\ZZ^n$-graded (i.e., we have the additive map $\MM(x) \to \ZZ^n$ defined by $\MM(x) \ni \be_i^x \longmapsto \be_i \in \NN^n \subset \ZZ^n$), 
and $\iota_{z,x}: \MM(x)\to \MM(z)$ is compatible with the $\ZZ^n$-grading, $\MM$ has the $\ZZ^n$-grading. 
If the degree of $\ua$ is $(a_1, \ldots, a_n) \in \ZZ^n$, then we have 
$$\pi_x(t^\ua)=\begin{cases}\prod_{i=1}^n t_i^{a_i} & (\text{if $\sigma(\ua) \le x$}),\\
0 & (\text{otherwise})
\end{cases}$$ in $S_x$.  

\begin{dfn}\label{M-graded}
An $A_P$-module $M$ is said to be {\it $\MM$--graded} if the following are satisfied.
\begin{itemize}
\item[(1)] $M=\bigoplus_{\ua\in \MM} M_{\ua}$ as $K$--vector spaces.
\item[(2)] For $\ua,\ub\in\MM$, we have 
$$t^{\ua}M_{\ub}\subset \bigoplus_{x\in[\sigma(\ua)\vee \sigma(\ub)]} M_{(\ua+\ub)|_x}.$$
Hence, if $[\sigma(\ua)\vee \sigma(\ub)]=\emptyset$, then $t^\ua M_\ub=0$.
\end{itemize}
\end{dfn}

For example, $A_P$ and $S_x$ for $x \in P$ are $\MM$-graded. 
Clearly, an $\MM$-graded $A_P$-module has the natural $\ZZ^n$-graded $A_P$-module structure. 
For $\MM$-graded $A_P$-modules $M,N$, we say an $A_P$-homomorphism $f:M \to N$ is {\it $\MM$-graded}, if $f(M_\ua) \subset N_\ua$ for all $\ua \in \MM$. Note that  an $\MM$-graded homomorphism is $\ZZ^n$-graded.

If $M$ is an $\MM$--graded module, then 
$$M_{\nleq x}:=\bigoplus_{\ua\notin \MM(x)} M_\ua$$
is an $\MM$--graded submodule of $M$ for all $x\in P$, and 
$$M_{\leq x}:=M/M_{\nleq x}$$
is a $\ZZ^{\rank(x)}$--graded module over $S_x\cong K[t_i\mid y_i\leq x]$.

Let $R=K[x_1, \ldots, x_m]$ be a polynomial ring with the standard $\NN^m$-grading (i.e., $\deg x_i =\be_i \in \NN^m$, where $\be_i \in \NN^m$ is the $i$-th standard vector). 
We recall the following notion introduced by the second author in \cite{Y00}. 
We say a finitely generated $\NN^m$-graded $R$-module $M$ is {\it squarefree} if the following conditions are satisfied.
\begin{itemize}
\item[(1)] $M$ is $\NN^m$-graded, that is, $M =\bigoplus_{\ba \in \NN^m} M_\ba$, 
\item[(2)] For all $\ba =(a_1, \ldots, a_m) \in \NN^m$ and all $1 \le i \le m$ with $a_i >0$, the multiplication map $M_\ba \ni y \longmapsto x_i y \in M_{\ba +\be_i}$ is bijective. 
\end{itemize}

\begin{dfn}
We say that an $\MM$--graded $A_P$--module $M$ is {\it squarefree} (over $P$), if $M_{\leq x}$ is a squarefree module over the polynomial ring $S_x\cong K[t_i\mid y_i\leq x]$ for all $x\in P$.
\end{dfn}

It is easy to see that $A_P$ and $S_x$ for $x \in P$ are squarefree. Clearly, a squarefree module $M$ is finitely generated. For $x \in P$ and $\ua \in \MM$ with $\sigma(\ua) =x$, there is a finite dimensional $K$-vector space $M_x$  
with $M_x \cong M_{\ua}$. 
Let $\Sq A_P$ be the category of  squarefree modules and $\MM$-graded homomorphisms between them. 

\begin{lem}\label{ZZ^n-graded ==> MM-graded}
For $M,N \in \Sq A_P$, a $\ZZ^n$-graded $S$-homomorphism $f:M \to N$ is also $\MM$-graded. Hence $\Sq A_P$ is a full subcategory of $\*Mod S$.     
\end{lem}

\begin{proof}
For $\ua \in \MM$, set  $$\{\, \ub \in \MM \mid \deg(\ua)=\deg(\ub) \, \} = \{\, \ua_0 (=\ua), \ua_1, \ldots, \ua_m \, \}.$$     
For $\alpha  \in M_\ua$, we have 
$$f(\alpha)=\sum_{i=0}^m\beta_i \quad \text{with} \quad \beta_i \in N_{\ua_i}.$$
Set $x_i:=\sigma(\ua_i) \in P$ for each $i$. 
It suffices to show that $\beta_i =0$ for all $i \ge 1$. 
For contradiction, assume that $\beta_j \ne 0$ for some $j \ge 1$.
Since $[x_0 \vee x_j] =\emptyset$, we have  $t_{x_j} \alpha=0$. 
So $t_{x_j}f(\alpha)=f(t_{x_j}\alpha)=f(0)=0$. 
On the other hand, since $N$ is squarefree, we have $t_{x_j} \beta_j \ne 0$. Since $[x_i \vee x_j] =\emptyset$ for $i \ne j$, we have 
$$t_{x_j}f(\alpha)=t_{x_j} \sum_{i=0}^m\beta_i =  \sum_{i=0}^m t_{x_j}\beta_i =t_{x_j} \beta_j\ne 0.$$
This is a contradiction. 
\end{proof}

By the above lemma, regarding $A_P$-modules as $S$-modules, $\Sq A_P$ is a full subcategory of $\*Mod S$.

\begin{prop}[{\cite[Proposition~3.3]{Y11}}]\label{Sq A_P}
The category $\Sq A_P$ is an abelian category with enough injectives, and an indecomposable injective is isomorphic to $S_x$ for some $x \in P$.     
\end{prop}

\begin{prop}
Let $M$ be a $\ZZ^n$--graded $S$--module. Then $M$ is a squarefree module over $P$ if and only if it admits a $\ZZ^n$-graded presentation
\begin{eqnarray}\label{sq-presentation}
0\to M \too \bigoplus_{x \in P} S_x^{l_x} \too  \bigoplus_{x \in P} S_x^{m_x}
\end{eqnarray}
for some $l_x, m_x \in \NN$.
\end{prop}
\begin{proof}
The assertion follows from Lemma~\ref{ZZ^n-graded ==> MM-graded} and Proposition~\ref{Sq A_P}. 
\end{proof}



For $M=\bigoplus_{\ba \in \ZZ^n} M_\ba \in \*Mod S$, set $M_{\ge \b0}:= \bigoplus_{\ba \in \NN^n} M_\ba \in \*Mod S$. We define the functor 
$\Nf :\*Mod S \to \*Mod S$ by 
$$\Nf(M):=[ \gHom_S(A_P,M)]_{\ge \b0}.$$ 
Since $[ \gHom_S(A_P,M)]_{\ge \b0} \cong (0:_M I_P)_{\ge \b0}$, we regard $\Nf(M)$ as a submodule of $M$. 
Since the functor $\gHom_S(A_P,-)$ is left exact, so is $\Nf(-)$.

\begin{lem}\label{extension to E} 
For $E =\bigoplus_{x \in P} \EE_x^{l_x}$,  
$E' =\bigoplus_{x \in P} \EE_x^{m_x}$ and $E'' =\bigoplus_{x \in P} \EE_x^{n_x}$, 
the following hold.
\begin{itemize}
\item[(1)] $\Nf(E)=\bigoplus_{x \in P} S_x^{l_x}$ and 
$[E]_{\ge \b0}=\Nf(E) \oplus \Lf(E)$. 
\item[(2)] A $\ZZ^n$-graded $S$-homomorphism $\varphi: \Nf(E)\to \Nf(E')$ can be uniquely extended to a clean homomorphism $\widetilde{\varphi}: E\to E'$.   
\item[(3)] For $\ZZ^n$-graded $S$-homomorphisms  $\varphi: \Nf(E)\to \Nf(E')$ and $\psi: \Nf(E')\to \Nf(E'')$, we have 
$(\psi \circ \varphi)^\sim=\widetilde{\psi}\circ \widetilde{\varphi}$ with the notation of (2). 
\end{itemize}   
\end{lem}

\begin{proof}
(1) and (2) are easy. (3) Since the composition of clean maps is clean, $\widetilde{\psi}\circ \widetilde{\varphi}:E\to E''$ is clean. Therefore, the assertion follows from the uniqueness given in (2). 
\end{proof}

Let $\*Mod A_P$ be the category of $\ZZ^n$-graded $A_P$-modules, and $\gModN A_P$ its full subcategory consisting of $\NN^n$-graded modules. Then $\gModN A_P$ is an abelian subcategory of $\*Mod A_P$, closed under kernels, cokernels, and extensions.

\begin{lem}\label{graded enough injective}
The category $\gModN A_P$ has enough injectives.
For any $x\in P$ and $\ba\in\ZZ^n$,
$\Nf(\EE_x(\ba))$ is an injective object in $\gModN A_P$.
In particular, $S_x$ is injective in $\gModN A_P$.
\end{lem}

\begin{proof}
We first show that $\Nf(E)$ is injective in $\gModN A_P$
for every injective module $E$ in $\*Mod S$.
Take an injection  
$f \in [\gHom_{A_P}(M,N)]_{\b0}=[\gHom_S(M,N)]_{\b0}$,
and a morphism 
$$g \in [\gHom_S(M,\Nf(E))]_{\b0}=[\gHom_{A_P}(M,\Nf(E))]_{\b0}.$$ 
Note that $\Nf(M)=M$, $\Nf(N)=N$ and $\Nf(f)=f$.
Since $E$ is injective in $\*Mod S$,
there exists a morphism $g' \in [\gHom_{S}(N,E)]_{\b0}$
such that $g'\circ f=g$.
 Therefore, we have $\Nf(g') \in [\gHom_{A_P}(N,\Nf(E))]_{\b0}$ and $g=\Nf(g')\circ f$. 
Thus $g'$ induces a morphism $N\to\Nf(E)$ extending $g$.
Hence $\Nf(E)$ is injective in $\gModN A_P$.

Now let $M\in\gModN A_P$.
Since $\*Mod S$ has enough injectives,
there exists a graded embedding $M\hookrightarrow E$
with $E$ injective in $\*Mod S$.
Since $\Nf(M)=M$, this embedding factors through $\Nf(E)$.
Therefore, $\gModN A_P$ has enough injectives.

Finally, $\EE_x(\ba)$ is injective in $\*Mod S$
for every $x\in P$ and $\ba\in\ZZ^n$.
Hence $\Nf(\EE_x(\ba))$ is injective in $\gModN A_P$.
Taking $\ba=\b0$ and using $\Nf(\EE_x)=S_x$,
we obtain the last assertion.
\end{proof}

\begin{lem}\label{squarefree-extension}
Let $0 \to M \xrightarrow{f} N \xrightarrow{g} L \to 0$ be a short exact sequence of $\ZZ^n$-graded $A_P$-modules. If $M$ and $L$ are squarefree, then so is $N$.
\end{lem}

\begin{proof}
We assume that $M$ and $L$ are squarefree modules with presentations as in \eqref{sq-presentation},   
$0\to M\xrightarrow{d_M} I_M\xrightarrow{d_M'} I_M'$ and $0\to L\xrightarrow{d_L} I_L\xrightarrow{d_L'} I_L'$.  
By Lemma \ref{graded enough injective}, these are injective presentations also in $\*Mod_{\geq \b0} A_P$. 
We also remark that $N \in \*Mod_{\geq \b0} A_P$.
By the Horseshoe Lemma, we have the presentation $0\to N\xrightarrow{d_N} I_N\xrightarrow{d_N'} I_N'$ for $I_N:=I_M\oplus I_L$ and $I_N':=I_M'\oplus I_L'$. Therefore, $N$ is squarefree.
\end{proof}

\begin{prop}
The derived category $D^b(\Sq A_P)$ can be seen as a full subcategory of $D^b(\*Mod A_P)$, and we have the category equivalence $\theD^b(\Sq A_P) \cong \theD^b_{\Sq A_P} (\*Mod A_P)$.
\end{prop}

Here $\theD^b_{\Sq A_P} (\*Mod A_P)$ denotes the full subcategory of $\theD^b(\*Mod A_P)$ consisting of $C^\bullet$ such that $H^i(C^\bullet) \in \Sq A_P$ for all $i$. We use analogous notation for other pairs of abelian categories.

\begin{proof}
The category $\gModN A_P$ has enough projectives, and a projective object, which is a  (possibly infinite) direct sum  of $A_P(-\ba)$ for $\ba\in \NN^n$, is also projective in $\*Mod A_P$. By the exercise right after \cite[Proposition 4.8]{Ha},  $$D^b_{\gModN A_P}(\*Mod A_P)\cong D^b(\gModN A_P)$$ and $D^b_{\gModN A_P}(\*Mod A_P)$ is a full subcategory of $D^b(\*Mod A_P)$. 

On the other hand, $\Sq A_P$ is an abelian full subcategory of $\gModN A_P$ closed under kernels, cokernels, and extensions. Moreover, it has enough injectives, and injective objects (i.e., finite direct sums of $S_x$'s for $x\in P$) are also injective in $\gModN A_P$.
By \cite[Proposition 4.8]{Ha}, we have $D^b_{\Sq A_P}(\gModN A_P)\cong D^b(\Sq A_P)$ and the left side is a full subcategory of $D^b(\gModN A_P)$. So the assertion follows.
\end{proof}

For $\bw:=\sum_{x \in P^*} \deg t_x \in \NN^n$, $\omega_S:=S(-\bw)$ is the $\ZZ^n$-graded canonical module of $S$.

\begin{lem}\label{simple in sq}
Let $\omega_{S_x}:= S_x(-\deg t_x)$ be the $\ZZ^n$-graded canonical module of $S_x$. 
\begin{enumerate}
\item A squarefree $A_P$-module $M$ is simple in $\Sq A_P$ (i.e. there is no non-trivial squarefree submodule) if and only if $M\cong \omega_{S_x}$
for some $x\in P$.
\item An $\MM$-graded $A_P$-module $M$ is squarefree if and only if there is a chain of $\MM$-graded submodules $0=M_0\subset M_1\subset \cdots \subset M_t=M$ such that each quotient $M_i/M_{i-1}$ is isomorphic to $\omega_{S_x}$ for some $x\in P$.
\end{enumerate}
\end{lem}

\begin{proof}
Let $\Lambda$ be the incidence algebra of the poset $P$, and $\mdL$ the category of finitely generated left $\Lambda$--modules. By \cite{Y11}, $\Sq A_P$ and $\mdL$ are equivalent. 

(1) For $\ua \in\MM$, we have   
$$[\omega_{S_x}]_\ua =\begin{cases}
K & (\text{if $\sigma(\ua)=x$})\\
0 & (\text{otherwise}.)
\end{cases}.$$ 
Since $\omega_{S_x}$ corresponds to a simple module in $\mdL$, we obtain (1). 

(2) Every object in $\mdL$ admits a composition series, hence $\Sq A_P$ has this property. 
Now the assertion follows from (1). 
Conversely, repeated application of Lemma \ref{squarefree-extension} shows that any module admitting such a filtration is squarefree.
\end{proof}

\begin{lem}\label{Hom(M,E)}
For $M \in \Sq A_P$ and $x \in P$,
$$[\gHom_S(M, \EE_x)]_{\ge \b0}=(M_x)^* \otimes_K S_x,$$
where $(M_x)^*$ is the dual $K$-vector space of $M_x$ (but we set the degree of $(M_x)^*$ to be 0). 
\end{lem}

\begin{proof}
The proof is the same as that of \cite[Lemma~3.20]{Y01a} with the obvious changes (replacing $M_F$ by $M_x$, and $\supp_+(-)$ by $\sigma(-)$ etc.).
\end{proof}

\begin{lem}\label{Ext is sq}
If $M \in \Sq A_P$, then $\gExt_S^i(M,\omega_S) \in \Sq A_P$. 
\end{lem}

In the rest of this paper, let $p:=\dim S= \#P^*$. 

\begin{proof}
Although Lemmas~\ref{squarefree-extension} and  \ref{simple in sq}  provide a shorter proof, we use Lemma~\ref{Hom(M,E)} here, as it gives more insight into the results of \S 7.

Let $T =K[t_1, \ldots, t_n]$ be the subring of $S$, and $\omega_T:=T(-\11)$ with $\11=(1, \ldots, 1) \in \ZZ^n$ the canonical module of $T$. It is easy to see that $M$ is a squarefree $T$-module. Hence $\gExt_T^i(M,\omega_T)$ is a squarefree $T$-module by \cite[Theorem~2.6]{Y00}, in particular, it is $\NN^n$-graded. Since 
$$\gExt_S^{p-i}(M, \omega_S) \cong \gExt_T^{n-i}(M,\omega_T)$$ by the $\ZZ^n$-graded local duality theorem, it follows that $\gExt_S^{i}(M, \omega_S)$ is $\NN^n$-graded.    

Let $\fp$ be a $\ZZ^n$-graded prime ideal of $S$.
Regarding $M$ as an $S$-module, we have $\operatorname{Ann}(M) \supset I_P$. Hence, $\fp \not \supset I_P$ implies $\gHom_S(M, \EE_S(S/\fp))=0$. If $\fp  \supset I_P$, equivalently, $\fp=\fp_x$ (and $\EE_S(S/\fp)= \EE_x$) for some $x \in P$,  $[\gHom_S(M, \EE_x)]_{\ge \b0}$ is squarefree by Lemma~\ref{Hom(M,E)}. 

Let $\DB_S$ be the $\ZZ^n$-graded residual complex of $S$. 
Then $\IB:={}^*\! D_S^{\bullet -p}$ is a minimal graded injective  resolution of $\omega_S$.  We have 
$${}^*\! I^{i}_S= \bigoplus_{\substack{\fp:\text{$\ZZ^n$-graded prime} \\ \dim S/\fp =p-i}} (\EE_S(S/\fp))(-\ba_\fp)$$
for some $\ba_\fp \in \ZZ^n$, and $\ba_{\fp_x}=\b0$ holds for all $x \in P$ by \cite[Proposition~4.3]{SY}. 
Hence 
$$[\gHom_S^i(M, \IB)]_{\ge \b0}= \bigoplus_{\substack{ x \in P \\ \rank x =p-i}} (M_x)^* \otimes_K S_x$$
by Lemma~\ref{Hom(M,E)}, and $[\gHom_S^\bullet(M, \IB)]_{\ge \b0}$ is a complex of squarefree modules. Since $\gExt_S^i(M, \omega_S) =H^i(\gHom_S^\bullet(M, \IB))
\cong [H^i(\gHom_S^\bullet(M, \IB))]_{\ge \b0}$, the assertion follows. 
%
\end{proof}

\section{Squarefree-Straight correspondence}
\begin{dfn}
Let $M$ be a $\ZZ^n$--graded $S$--module. If it admits a $\ZZ^n$-graded presentation
\begin{equation}\label{straight presentation}
p_M : 0\too M \xrightarrow{d_M} E_M \xrightarrow{d'_M} E_M',
\end{equation}
where $E_M$ and $E'_M$ are finite direct sums of $\EE_x$'s for various $x \in P$ and $d'_M$ is clean, 
then we say the pair $(M,p_M)$ is a {\it straight module} (over $P$). 
\end{dfn}

Let $(M, p_M)$ be a straight module with the presentation $p_M$ as in \eqref{straight presentation}. We have 
$E_M = \bigoplus_{x \in P} \EE_x^{l_x}$ and $E_M' =\bigoplus_{x \in P} \EE_x^{m_x}$ for some $l_x, m_x \in \NN$.
Then $\Nf(E_M) \cong \bigoplus_{x \in P} S_x^{l_x}$ and  $\Nf(E_M') \cong \bigoplus_{x \in P} S_x^{m_x}$ hold, and $p_M$ yields the exact sequence 
$$ 0\too \Nf(M) \xrightarrow{\Nf(d_M)} \Nf(E_M) \xrightarrow{\Nf(d'_M)} \Nf(E_M'),$$
and hence $\Nf(M)$ is a squarefree module.

\begin{dfn}\label{str-map}
Let $(M,p_M)$ and $(N,p_N)$ be straight modules over $P$. 
A {\it clean map} $(M, p_M) \to (N,p_N)$ is an equivalence  class of commutative diagrams 
\begin{equation}\label{hom in Str}
  \begin{CD}
    p_M: 0 @>>>  M  @>{d_M}>>  E_M  @>_{d_M'}>>  E_M'  \\
    @.     @VV{\varphi}V  @VV{\varphi^0}V  @V{\varphi^1}VV   @. \\
    p_N: 0 @>>>  N @>{d_N}>>  E_N @>{d_N'}>>  E_N' ,
  \end{CD}
\end{equation}
where $\varphi^0$ and $\varphi^1$ are clean maps. 
Two such diagrams are equivalent if their components $M \to N$ coincide. 
So we simply say that $\varphi\in  [\gHom_S (M, N)]_\b0$ is clean. 
\end{dfn}

\begin{rem}\label{remark on clenness}
Let $x,z\in P$ with $x>z$, and $\varphi\in  [\gHom_S (\EE_x, \EE_z)]_\b0$  a non-clean map with $\varphi(1_x)=1_z$. By \cite[Theorem~3.13]{SY}, there is a $\ZZ^n$--graded automorphism $\sigma$ of $\EE_x$ such that $\varphi\circ \sigma$ is clean. Thus, if we take $0 \to \EE_x \stackrel{\sigma^{-1}}{\too} \EE_x \to 0$ and $0 \to \EE_z\stackrel{\mathrm{id}}{\too} \EE_z \to 0$ as the presentations of $\EE_x$ and $\EE_z$, then $\varphi$ induces a clean map. In fact, we have the commutative diagram
$$
  \begin{CD}
  0 @>>>  \EE_x  @>{\sigma^{-1}}>>  \EE_x  @>>> 0  \\
    @.     @VV{\varphi}V  @VV{\varphi \sigma}V  @VVV   @. \\
   0 @>>>  \EE_z @>{\mathrm{id}}>>  \EE_z @>>>  0.  
  \end{CD}
$$
Thus, a naive extension of the definitions from the previous paper \cite{SY} may lead to a slight mismatch. 
So, when we regard $\EE_x$ as a straight module, we use $0 \too \EE_x \stackrel{\operatorname{id}}{\too} \EE_x \too 0$
as its presentation. 
\end{rem}

Let $\Str_P S$ be the category consisting of straight modules and clean morphisms between them. 

For a straight module $(M, p_M)$,  working directly with the module $M$ itself is practically sufficient or intuitively natural (see Remark~\ref{non-full subcategory} below). 
For this reason, we often abbreviate the pair $(M, p_M)$ simply as $M$. Even in such cases, the reader should keep in mind that the specific presentation $p_M: 0 \too M \xrightarrow{d_M}   E_M  \xrightarrow{d'_M}   E'_M$ is implicitly fixed. 
Unless otherwise specified, as $p_M$, we use the minimal one. 
That is, $E_M$ is a $^*$injective envelope of $M$, and $E'_M$ is a $^*$injective envelope of $E_M/M$.

\begin{lem}\label{ess extension}
A straight module $M$ is a *essential extension (i.e., essential extension in the $\ZZ^n$-graded context) of  $\Nf(M)$. 
\end{lem}

\begin{proof}
With the notation of \eqref{straight presentation}, $M$ can be seen as a submodule of $E_M=\bigoplus_{x\in P} \EE_x^{l_x}$.  
Since  $\EE_x$ is a  *essential extension of $\Nf(\EE_x)=S_x$, 
$E_M$ is  a  *essential extension of $\Nf(E_M)$.  
We will show that $M'\cap \Nf(M)\neq 0$ for any graded submodule $M'\neq 0$ of $M$. Regarding $M'\subset M\subset E_M$, we have $M'\cap \Nf(E_M)\neq 0$. Since $M'\cap \Nf(E_M)\subset \Nf(M)$, the assertion follows.
\end{proof}

\begin{lem}\label{existbase}
Let $M$ be a straight module with the presentation \eqref{straight presentation}. 
For $\ba \in \NN^n$ and $\alpha \in M_\ba$, we have $\alpha_1, \alpha_2 \in M_\ba$ with $\alpha=\alpha_1+\alpha_2$, $d_M(\alpha_1) \in \Nf(E_M)$ (equivalently, $\alpha_1 \in \Nf(M)$), 
and $d_M(\alpha_2) \in \Lf(E_M)$. 
\end{lem}

\begin{proof}
We can take $\beta_1\in [\Nf(E_M)]_\ba$ and $\beta_2 \in [\Lf(E_M)]_\ba$ so that $d_M(\alpha)=\beta_1+\beta_2$. 
Since $d'_M(\beta_1) \in \Nf(E'_M)$, $d'_M(\beta_2) \in \Lf(E'_M)$ (for the latter, we use the cleanness of $d_M'$), and $$d'_M(\beta_1)+d'_M(\beta_2)= d'_M(\beta_1+\beta_2)=d'_M(d_M(\alpha))=0,$$ we have $d'_M(\beta_1)=d'_M(\beta_2)=0$. 
By the exactness of \eqref{straight presentation}, there is $\alpha_1 \in M_\ba$ with $d_M(\alpha_1)=\beta_1$. 
For $\alpha_2:=\alpha-\alpha_1 \in M_\ba$, we have  
$d_M(\alpha_2)=\beta_2 \in \Lf(E_M).$  
Hence $\alpha=\alpha_1 +\alpha_2$ is a desired decomposition.  
\end{proof}

The assignment $\Str_P S \ni (M, p_M) \longmapsto \Nf(M) \in \Sq A_P$ defines a functor $\Str_P S \to \Sq A_P$, which we shall also denote by $\Nf$. 

\begin{thm}\label{sq-str}
The functor $\Nf
:\Str_P S \to \Sq A_P$ gives the equivalence $\Str_P S \simeq \Sq A_P$. Hence $\Str_P S$ is an abelian category. 
\end{thm}

\begin{proof}
First we will check that $\Nf$ is essentially surjective. 
For any $N\in \Sq A_P$, there is a presentation
  $$0 \to N \to \bigoplus_{x\in P}S_x^{l_x} \xrightarrow{\varphi} \bigoplus_{x\in P}S_x^{m_x}.$$
By Lemma~\ref{extension to E}, we have
$$p_K: 0\to \Ker \widetilde{\varphi} \to \bigoplus_{x\in P}\EE_x^{l_x} \xrightarrow{\widetilde{\varphi}} \bigoplus_{x\in P}\EE_x^{m_x},$$
where $\widetilde{\varphi}$ is clean and $\Nf(\widetilde{\varphi})=\varphi$. Then $\Nf(\Ker \widetilde{\varphi})\cong N$. Therefore, $\Nf$ is essentially surjective because $(\Ker \widetilde{\varphi}, p_K)$ is straight.

We will check that $\Nf$ is fully faithful. That is, 
\begin{equation}\label{N for Hom}
\Nf:\Hom_{\Str_P S}(M,N)\to  [{}^* \! \Hom_S (\Nf(M), \Nf(N))]_\b0
\end{equation}
is bijective for each $M, N\in \Str_P S$. 

First, we check the surjectivity of \eqref{N for Hom}. Take $\psi \in  [{}^* \! \Hom_S(\Nf(M), \Nf(N))]_\b0$.
Since $M$ and $N$ are straight modules, there exist presentations
  $$p_M: 0 \to M \xrightarrow{d_M} E_M \xrightarrow{d'_M} E_M'$$
and 
$$p_N: 0 \to N \xrightarrow{d_N} E_N \xrightarrow{d'_N} E_N'.$$
Since $\Nf(E_M), \Nf(E'_M), \Nf(E_N)$ and $\Nf(E'_N)$ are injective in $\Sq A_P$, 
there exists the following commutative diagram
\[
  \begin{CD}
     0 @>>>  \Nf(M)  @>{\Nf(d_M)}>>  \Nf(E_M)  @>{\Nf(d_M')}>>  \Nf(E_M')  \\
    @.     @VV\psi V  @VV\psi^0 V  @VV\psi^1 V   @. \\
     0 @>>>  \Nf(N) @>{\Nf(d_N)}>>  \Nf(E_N) @>{\Nf(d'_N)}>>  \Nf(E_N') .
  \end{CD}
\]
By Lemma~\ref{extension to E} (2) and (3), there exist clean maps $\widetilde{\psi^0}$ and $\widetilde{\psi^1}$  such that the right square of the following diagram commutes 
\[
  \begin{CD}
     0 @>>>  M  @>{d_M}>>  E_M  @>{d'_M}>>  E_M'  \\
    @.    @VV{\widetilde{\psi}}V   @VV{\widetilde{\psi^0}}V  @VV{\widetilde{\psi^1}}V   @. \\
     0 @>>>  N @>{d_N}>>  E_N @>{d'_N}>>  E_N' ,
  \end{CD}
\]
and $\widetilde{\psi^0}$ induces $\widetilde{\psi}: M\to N$, which is a clean map with $\Nf(\widetilde{\psi})=\psi$.

Next, we check the injectivity of \eqref{N for Hom}.
Consider the  commutative diagram \eqref{str-map}. 
For a contradiction, assume that $\varphi \ne 0$ and 
$\Nf(\varphi)=0$. Then
\[
  \begin{CD}
     0 @>>>  \Nf(M)  @>{\Nf(d_M)}>>  \Nf(E_M)  @>>>  \Nf(E_M')  \\
    @.     @VV{\Nf(\varphi)=0}V  @VV{\Nf(\varphi^0)}V  @VVV   @. \\
     0 @>>>  \Nf(N) @>{\Nf(d_N)}>>  \Nf(E_N) @>>>  \Nf(E_N') ,
  \end{CD}
\]
and $\Nf(\varphi^0\circ d_M)=\Nf(\varphi^0)\circ \Nf(d_M)=0$. Since $\varphi \ne 0$ now, we have $\varphi^0\circ d_M=d_N \circ \varphi \neq 0$, and  $0\neq \varphi^0(d_M(\alpha))\in \Nf(E_N)$ for some homogeneous $\alpha\in M$ by Lemma~\ref{ess extension}. 
Note that $\deg \alpha \ge \b0$. 
By Lemma \ref{existbase}, we can take $\alpha_1, \alpha_2 \in M$ such that $\alpha=\alpha_1+\alpha_2$, $d_M(\alpha_1) \in \Nf(E_M)$ and $d_M(\alpha_2) \in \Lf(E_M)$. 
Clearly, $\varphi^0(d_M(\alpha_1)) \in \Nf(E_N)$. 
Since $\varphi^0$ is clean, we have $\varphi^0(d_M(\alpha_2)) \in \Lf(E_N)$.   
On the other hand, since $\Nf(\varphi^0\circ d_M)=0$ now and $\alpha_1 \in \Nf(M)$, we have $\varphi^0(d_M(\alpha_1))=0$. Hence we have 
$$
\Nf(E_N) \ni \varphi^0(d_M(\alpha))=\varphi^0(d_M(\alpha_1+\alpha_2))=\varphi^0(d_M(\alpha_2)) \in \Lf(E_N).
$$
It implies that $\varphi^0(d_M(\alpha))=0$, and 
this is a contradiction. 
\end{proof}


\begin{dfn}[cf. \cite{M00}]
Let us denote a quasi-inverse of $\Nf: \Str_P S\to \Sq A_P $ as $\C: \Sq A_P\to \Str_P S$. We call $\C$ the \v{C}ech hull.
\end{dfn}

\begin{rem}\label{non-full subcategory}
Let $(M,p_M), (N, p_N)$ be straight modules. Then $(M,p_M)$ and $(N,p_N)$ are isomorphic in $\Str_P S$ if and only if $M$ and $N$ are also isomorphic in $\*Mod S$. 
In fact, if $M \cong N$ in $\*Mod S$, then we get $\Nf(M) \cong \Nf(N)$, and it induces an isomorphism between $(M,p_M)$ and $(N,p_N)$ via $\C$. The converse is easier.
Because kernels and cokernels in $\Str_P S$ are isomorphic to those in $\*Mod S$.
The forgetful functor $\Str_P S \to \*Mod S$ is faithful, but need not be full.

This situation is conceptually analogous to the familiar example of the simplex category $\Delta$, whose objects are finite totally ordered sets and whose morphisms are order-preserving maps: the forgetful functor from $\Delta$ to the category of sets  is faithful but not full.
\end{rem}

By Theorem~\ref{sq-str}, $\Str_P S$ is abelian, but it is far from obvious that kernels, images, and cokernels agree with those computed in $\*Mod S$. 
In Theorems~\ref{Ker is straight} and \ref{Coker is straight} below, we will show that $\Ker \varphi$, $\Im\varphi$ and $\Coker\varphi$ are straight, but they are no more than  $\ZZ^n$-graded $S$-modules at this stage. 

\begin{rem}[{cf. \cite[Example 2.3.19]{Ye}}]
Let $X$ be a topological space. Then sheaves of abelian groups on $X$ form a full subcategory of the category of presheaves of abelian groups on $X$. Both categories are abelian. However, the cokernel of a morphism in the former category need not be isomorphic to its cokernel in the latter. By passing to the opposite categories, we obtain an example in which the kernels of a morphism in the two categories need not be isomorphic.
\end{rem}

\begin{lem}\label{N-exact-Ker}
If $\varphi\in \Hom_{\Str_P S}(M,N)$, then 
$\Nf(\Im\varphi)=\Im\Nf(\varphi)$. In particular, 
$$0\to \Nf(\Ker\varphi) \to \Nf(M) \to \Nf(\Im\varphi)\to 0$$
is exact (in the category $\*Mod S$).
\end{lem}

\begin{proof}
We use the same notation as \eqref{hom in Str}. 
Since $\Nf$ is left exact, it suffices to show that $\Nf(\varphi): \Nf(M)\to \Nf(\Im\varphi)$ is surjective.  Take an arbitrary homogeneous element $\varphi(\alpha)\in \Nf(\Im\varphi)$ for $\alpha \in M$. 
We may assume that $\alpha$ is also homogeneous (in what follows, we always choose homogeneous preimages in similar situations). Then $\deg \alpha \geq \mathbf{0}$.
We will check that there exists $\alpha'\in \Nf(M)$ such that $\varphi(\alpha')=\varphi(\alpha)$. 



 By Lemma \ref{existbase}, we can take $\alpha_1, \alpha_2 \in M$ such that $\alpha=\alpha_1+\alpha_2$, $d_M(\alpha_1) \in \Nf(E_M)$ (equivalently, $\alpha_1 \in \Nf(M)$) and $d_M(\alpha_2) \in \Lf(E_M)$. Clearly, 
$d_N(\varphi(\alpha_1))=\varphi^0(d_M(\alpha_1))\in \Nf(E_N).$ 
Since $\varphi^0$ is clean, we have 
$d_N(\varphi(\alpha_2))=\varphi^0(d_M(\alpha_2))\in   \Lf(E_N).$   
Since 
$$\Nf(E_N) \ni d_N(\varphi(\alpha))=d_N(\varphi(\alpha_1))+d_N(\varphi(\alpha_2)),$$
we have $d_N(\varphi(\alpha_2))=0$ and 
$d_N(\varphi(\alpha))=d_N(\varphi(\alpha_1))$.  
Since $d_N$ is injective, we have  
$\varphi(\alpha)=\varphi(\alpha_1)$, and $\alpha_1 \in \Nf(M)$ is the desired element. 
\end{proof}

\begin{lem}\label{ess-Im}
For $0 \ne \varphi\in \Hom_{\Str_P S}(M,N)$, $\Im \varphi$ is a  ${}^*$essential extension of $\Nf(\Im\varphi)$. 
\end{lem}

\begin{proof}
Through the injection $d_N$, we regard $N$ as a submodule  of $E_N$. Let $L$ be a non-zero graded submodule of $\Im\varphi$. Since $N$ is a *essential extension of $\Nf(N)$, we have $\Nf(L)=L \cap \Nf(\Im\varphi)=L \cap \Nf(E_N) \ne 0$.
\end{proof}

\begin{thm}\label{Ker is straight}
For $\varphi\in \Hom_{\Str_P S}(M,N)$, let $\Ker\varphi$ (resp. $K_\varphi$) be the kernel of $\varphi$ in $\*Mod S$ (resp. $\Str_P S$). Then $\Ker \varphi = K_\varphi$ as submodules of $M$. In particular, $\Ker \varphi$ is straight.
\end{thm}

\begin{proof}
Set 
$$p_M: 0\to M\xrightarrow{d_M} E_M\xrightarrow{d'_M} E'_M  
\quad \text{and} \quad 
p_N:0\to N\xrightarrow{d_N} E_N\xrightarrow{d'_N} E'_N.$$ 
If $\varphi=0$, then $M= \Ker\varphi$. 
So we may assume that $\varphi\neq 0$.   
Let $p_K: 0 \to K_\varphi \to E_K \to E_K'$ be the presentation of $K_\varphi$ as a straight module,  and $k_\varphi:K_\varphi \to M$ the clean map given by the kernel inclusion. 
Note that $\varphi\circ k_\varphi=0$ (although $\varphi\circ k_\varphi=0$ is principally an equality in $\Str_P S$, we can easily see that it remains true even when viewed in $\*Mod S$), and hence $\Im k_\varphi \subset \Ker \varphi$.  

First, we will show that $k_\varphi$ is injective also in $ \*Mod S$. For a contradiction, assume that $\Ker k_\varphi \ne 0$. Since $K_\varphi$ is a *essential extension of $\Nf(K_\varphi)$, we have $\Nf(\Ker k_\varphi)=\Ker k_\varphi \cap \Nf(K_\varphi) \ne 0$.
However, since $\Nf:\Str_P S\to \Sq A_P$ gives an equivalence, $\Nf(k_\varphi):\Nf(K_\varphi)\to \Nf(M)$ is an injection, that is,  $\Nf(\Ker k_\varphi) =\Ker \Nf(k_\varphi) = 0$. This is a contradiction. So $k_\varphi$ is injective, and we regard $K_\varphi$  as a  submodule of $\Ker \varphi$ through $k_\varphi$. 
Since $\Nf$ gives the equivalence, we have $\Nf(\Ker \varphi)=\Ker(\Nf(\varphi))=\Nf(K_\varphi)$,  
and we see that $\Ker\varphi$ is a $^*$essential extension of $K_\varphi$. Hence, $E_K$ is also a  *injective envelope of $\Ker \varphi$, and there exists the following commutative diagram 
\begin{equation}\label{K-diagram}
  \begin{CD}
    0 @>>>  K_\varphi  @>{d_K}>>  E_K  @>{d'_K}>>  E'_K  \\
    @.     @VV{k_\varphi}V  @VV{\mathrm{id}_{E_K}}V  @VV{\tau}V   @. \\
     0 @>>>  \Ker\varphi @>{d}>>  E_K @>{d'}>>  E,
  \end{CD}
\end{equation}
where $E$ is the ${}^*$injective envelope of $E_K/\Im d$. 
We have
\begin{equation}\label{NK-diagram}
  \begin{CD}
     0 @>>>  \Nf(K_\varphi)  @>{\Nf(d_K)}>>  \Nf(E_K)  @>>>  \Nf(\Im d'_K) @>>> 0   \\
    @.     @VV{\Nf(k_\varphi)}V  @VV{\mathrm{id}_{\Nf(E_K)}}V  @VV{\Nf(\tau|_{\Im d'_K})}V   @. \\
     0 @>>>  \Nf(\Ker\varphi) @>>>  \Nf(E_K) @>>>  \Nf(\Im d'), @.  
  \end{CD}
\end{equation}
where the exactness of the first row follows from Lemma \ref{N-exact-Ker}. We will show that $\tau$ is injective. 
Applying the Snake Lemma to the commutative diagram \eqref{NK-diagram}, we see that $\Nf(\tau|_{\Im d'_K})$ is injective. 
Since $\Nf(\Im d'_K)\subset \Im d'_K$ (resp. $\Im d'_K \subset E'_K$) is a ${}^*$essential extension by Lemma \ref{ess-Im} (resp. the minimality of $p_K$), the extension $\Nf(\Im d'_K) \subset E'_K$ is also, and 
we see that $\tau$ is an injection by an argument similar to the above.  
Applying the Snake Lemma to the commutative diagram \eqref{K-diagram}
(more precisely, we replace $E'_K$ by $\Im d'_K$, and $\tau$
by its restriction), we see that $k_\varphi$ is an isomorphism.
\end{proof}

\begin{cor}\label{inj-mono}
For $\varphi\in \Hom_{\Str_P S}(M,N)$, it is injective in the usual sense if and only if it is a monomorphism in $\Str_P S$. 
\end{cor}

\begin{proof}
Since $\Ker \varphi = K_\varphi$ in the notation of Theorem~\ref{Ker is straight}, the assertion follows. 
\end{proof}

\section{Leading monomials}

Let $P_{\ge 2}$ be the subset of $P$ consisting of the elements with rank $\ge 2$, that is, 
$P_{\ge 2}= P^* \setminus \{y_1, \ldots, y_n\}.$ 
Let $T$ be the set of (Laurent) monomials in $$K[t_1^{\pm 1}, \ldots, t_n^{\pm 1}, t_z^{-1} \mid z \in P_{\ge 2}],$$ that is,    
\begin{eqnarray*}
T &:=& \left \{ \, \prod_{i=1}^n t_i^{a_i} \cdot \prod_{z \in P_{\ge 2} } t_z^{-a_z} \mid a_i \in \ZZ, a_z \in \NN \, \right \} \\
&=& \left \{ \prod_{x \in P^*} t_x^{a_x} \mid a_x \in \ZZ 
\text{ and $a_x \le 0$ for $x \in  P_{\ge 2}$} \right \}.
\end{eqnarray*} 
For a monomial $\alpha =\prod_{x \in P^*} t_x^{a_x} \in T$, set 
$$\alpha^-:= \prod_{a_x <0} t_x^{-a_x} \in S \quad \text{and} \quad 
\alpha^+:= \prod_{a_i > 0} t_i^{a_i} \in K[t_1, \ldots, t_n] \subset S.$$ 
Note that $\alpha=\alpha^+/\alpha^-$. 

 Fix a total order on the set of variables of $S$ (i.e., $\{ t_x \mid x \in P^*\}$) such that any variable $t_z$ for $z \in P_{\ge 2}$ is smaller than $t_1, \ldots, t_n$, and $t_1 > t_2 >\cdots >t_n$ holds. 
We equip $S$ with the lexicographic order $\lexico$  with respect to the above total order of the variables (see \cite[Chapter 2, \S 2,  Definition 3]{CLO}).  We define the total order $\succ$ on $T$ as follows. For $\alpha, \beta \in T$; 
$$
\alpha \succ \beta \, \stackrel{\text{def}}{\Longleftrightarrow}   \alpha^- \lexico \beta^-, \ \text{or} \ \alpha^-=\beta^- \  \text{and} \ 
\alpha^+ \lexico \beta^+.
$$

\begin{lem}
The above order $\succ$ on $T$ is a well-order, that is,  any non-empty subset of $T$ has a least element. \end{lem}

\begin{proof}
The assertion easily follows  from the fact that $\lexico$ is a well-order. 
\end{proof}

For $f=\sum_{i=1}^m r_i(u_i \otimes v_i) \in \EE_x$ such that $u_i \otimes v_i$ is a monomial, set 
$$T_f:=\{ u_iv_i \in T \mid r_i \ne 0\}.$$
Here the product $u_iv_i$ is computed in the Laurent polynomial ring. Since $T_f$ is a finite set, we have the largest element $\lm(f)$ of $T_f$.

 Let $E$ be a  finite direct sum of $\EE_x$'s for various $x \in P$.  
Then there are $x_1, \ldots, x_m \in P$ such that 
$$E=\bigoplus_{i=1}^m \EE_{x_i}$$
(we allow the case $x_i=x_j$). 
For $0 \ne f \in E$, we can take $f_i \in \EE_{x_i}$ for $1 \le i \le m$ such that $f=\sum_{i=1}^mf_i$. Set 
$$\lm(f):= \text{the largest element of $\{\lm(f_i) \mid 1 \le i \le m, f_i \ne 0  \}$ w.r.t. $\succ$}. $$

\begin{lem}
With the above situation, if $\lm(f) \in K[t_1, \ldots, t_n]$, then $f \in \Nf(E)$. 
\end{lem}

\begin{proof}
 By definition.    
\end{proof}

\begin{lem}\label{clean and lm}
  Let $E$ and $E'$ be finite direct sums of $\EE_x$'s for various $x \in P$, and $\varphi:E\to E'$ a clean map. 
If $\varphi(f) \ne 0$ for $f \in E$, we have 
$\lm(\varphi(f)) \preceq \lm(f)$. 
\end{lem}

\begin{proof}
It suffices to consider the case $E=\EE_x$ and $E'=\EE_{x'}$ for some $x, x' \in P$ and $f$ is a monomial. In this case, the assertion follows from the description of the clean map given by \cite{SY} (see the end of \S 2). 
\end{proof}

For a monomial $u \otimes v \in \EE_{x'}$ with 
$$uv=\prod_{x \in P^*} t_x^{a_x} \in T,$$ we say $t_x$  {\it appears with negative powers} in $u \otimes v$, if $a_x <0$. 
Let 
$$\nv(u \otimes v):= \{ t_x \mid x \in P^*, \text{$t_x$ appears with negative powers in $u \otimes v$}\}.$$ 
For 
$$f=\sum_{i=1}^mf_i \in \bigoplus_{i=1}^m \EE_{x_i},$$
we say $t_x$  appears with negative powers in $f$, if it appears  with negative powers in some term of some $f_i$, and  we define $\nv(f)$ in the natural way. 
Clearly, $\nv(f)=\emptyset$ if and only if $f \in \Nf(E)$. We also  set 
$$\lv(f):=\text{the largest variable of $\nv(f)$ w.r.t.  $\succ$}.$$
If $\nv(f)=\emptyset$, we set $\lv(f)=0$. For convenience, 0 is smaller than any monomials with respect to our order $\prec$. 
Clearly, $\lv(f) = 0$ if and only if $f \in \Nf(E)$. If $\lv(f) \ne 0$, then we have  $\lv(f)=\lv(\lm(f))$. 

For $x\in P$ and $z\in P_{\geq2}$, set
\begin{equation*}
 \wt_z^{\,x}
 :=
 \begin{cases}
 t_z-\displaystyle\prod_{y_i\leq z}t_i,
       &\text{if }z\leq x,\\[2mm]
 t_z, &\text{if }z\not\leq x.
 \end{cases}
\end{equation*}
For a monomial $u\otimes v\in\EE_x$, define
$$\delta_z^x(u\otimes v):=u\otimes t_zv.$$
We will extend $\delta_z^x$
$K$-linearly to $\EE_x$. Note that
\begin{equation}\label{delta-as-multiplication}
 \delta_z^x(f)=\wt_z^{\,x}f
 \qquad
 (f\in\EE_x).
\end{equation}
For a module $E=\bigoplus_j\EE_{x_j},$
we define
\[
 \delta_z^E:=\bigoplus_j\delta_z^{x_j} \in \operatorname{End}_K E.
\]
We also set
$$E^{\nf}
 :=
 \{f\in E\mid
   \nv(f)\cap\{t_1,\ldots,t_n\}=\emptyset
 \}.$$
Here, the superscript ``nf" stands for ``negative $t_i$ free".

\begin{lem}\label{delta-lower}
Let $z\in P_{\geq2}$ and $E$ be a finite direct sum of $\EE_x$'s.
Then the following hold.
\begin{enumerate}
\item
For $f\in E$, 
$\delta_z^E(f)=0$ if and only if $t_z\notin\nv(f).$
\item
If $f\neq0$ and $\lv(f)=t_z$, then we have 
$\delta_z^E(f)\neq0$,
$\lm(\delta_z^E(f))\prec\lm(f)$,
and  $\lv(\delta_z^E(f))\preceq t_z.$
\item
For $f\in E$, we have $\delta_z^E(f\cdot t_z^{-1})=f.$ 
If $f\neq0$ and $\lv(f)\preceq t_z$, then
$\lv(f\cdot t_z^{-1})=t_z.$ 
\end{enumerate}
\end{lem}

\begin{proof}
Easy. 
\end{proof}

\begin{lem}\label{delta-clean}
Let $\varphi:E\longrightarrow E'$ be a clean map between finite direct sums of $\EE_x$'s. For
$z\in P_{\geq2}$ and $f\in E^{\nf}$, we have
\[
\varphi(f)\in(E')^{\nf}
\]
and
\begin{equation}\label{delta-clean-eq}
 \delta_z^{E'}(\varphi(f))
 =
 \varphi(\delta_z^E(f)).
\end{equation}
\end{lem}

\begin{proof}
The first assertion is immediate from Lemma~\ref{clean and lm}, so it suffices to prove \eqref{delta-clean-eq} for a nonzero clean
map $\psi:\EE_x\longrightarrow \EE_{x'}$ with  $x\geq x'$. Set
$$p_z:=\prod_{y_i\leq z}t_i.$$
We consider the following three cases.

Assume first that $z\leq x'$. Then we have $z\leq x$ as well, hence $\wt_z^{\,x}=t_z-p_z=\wt_z^{\,x'}$, and we have
$$\psi(\delta_z^x(f))=\psi((t_z-p_z)f)=(t_z-p_z)\psi(f)=\delta_z^{x'}(\psi(f)).
$$
Next assume that $z\nleq x$. Then $z\nleq x'$, and $\widetilde t_z^{\,x}=t_z=\widetilde t_z^{\,x'}$, and the assertion follows.

It remains to consider the case $z\leq x$ and $z\not\leq x'$.
Since $x',z\leq x$, there is $y_i$ such that $y_i\leq z, y_i\not\leq x'$.
Since $\psi(f)\in(\EE_{x'})^{\nf}$, no monomial occurring in
$\psi(f)$ contains a negative power of $t_i$. Now $y_i\not\leq x'$, therefore $t_i\psi(f)=0$. Consequently, $p_z\psi(f)=0.$
Now we have $\widetilde t_z^{\,x'}=t_z$, and 
$$\psi(\delta_z^x(f))= \psi((t_z-p_z)f)=(t_z-p_z)\psi(f)= t_z\psi(f)=
 \delta_z^{x'}(\psi(f)).$$
\end{proof}

\begin{lem}\label{nf-change}
Let $E$ and $E'$ be finite direct sums of $\EE_x$'s for various $x \in P$, and $\varphi:E\to E'$ a clean map. 
For any homogeneous  $\alpha\in E$ with $\varphi(\alpha)\in (E')^{\nf}$, 
there exists a homogeneous element $\alpha'\in E^{\nf}$ such that $\varphi(\alpha)=\varphi(\alpha')$.
\end{lem}

\begin{proof}
Recall that $\EE_x =K[t_i^{\pm 1} \, | \, y_i \le x] \otimes_K K[t_z^{-1} \, | \, z \in P^*_{-x}]$. In this proof, we identify $\Nf(\EE_x)$ with $K[t_i \, | \, y_i \le x]$. The same is true for $E$ and $E'$. Set $B:=K[t_z^{-1}\mid z\in P_{\ge 2}]$. 
In the present situation, we have the canonical graded $K$--isomorphism $\EE_x^{\nf}\cong S_x\otimes_K B$ for $x\in P$, and hence 
$$E^{\nf}\cong \Nf(E)\otimes_K B, \quad (E')^{\nf}\cong \Nf(E')\otimes_K B.$$
By the description of a clean map (see the end of \S 2), we have $\varphi|_{E^{\nf}}=\Nf(\varphi)\otimes \id_B$. 

Set $\varphi(\alpha)=\sum_{i} u_i\otimes v_i$, where $v_i$ are distinct monomials in $B$, and $u_i\in \Nf(E')$. We claim that $u_i\in \Im\Nf(\varphi)$ for any $i$. 
 Suppose, for a contradiction, that $u_j\notin \Im\Nf(\varphi)$ for some $j$. Since $\Im\Nf(\varphi)\in \Sq A_P$ with $\Im\Nf(\varphi)\subset \Nf(E')$, there exists a finite direct sum $J:=\bigoplus_{x\in P} S_x^{r_x}$ and a graded homomorphism $g:\Nf(E')\to J$ such that 
$$0\to \Im\Nf(\varphi)\to \Nf(E')\xrightarrow{g} J$$
is exact, and we have $g(\Im\Nf(\varphi))=0$ and $g(u_j)\neq 0$. 
Now we have the clean map $\C(g): E'\to \C(J)$. 
Since $g\circ \Nf(\varphi)=0$, two clean maps $\C(g)\circ  \varphi$ and $0:E\to \C(J)$ have the same restriction to $\Nf(E)$. Since $\Nf(-)$ gives an equivalence, we have $\C(g)\circ \varphi=0$. In particular, $\C(g)(\varphi(\alpha))=0$.

On the other hand, since  $\varphi|_{E^{\nf}}=\Nf(\varphi)\otimes \id_B$ and  $\C(g)|_{{(E')}^{\nf}}=g\otimes \id_B$, we have $$0=\C(g)(\varphi(\alpha))=\sum g(u_i)\otimes v_i.$$
Hence $g(u_i)=0$ for any $i$. This contradicts $g(u_j)\neq 0$. Therefore, $u_i\in \Im\Nf(\varphi)$ for any $i$.
For any $i$, choose a homogeneous $u'_i\in \Nf(E)$ such that $\Nf(\varphi)(u'_i)=u_i$.
Set $\alpha':=\sum_i u'_i\otimes v_i$. Then $\alpha'\in \Nf(E)\otimes_K B=E^{\nf}$, and 
\[
\varphi(\alpha')=\sum_i \varphi(u'_i)\otimes v_i
=\sum_i u_i\otimes v_i
=\varphi(\alpha).
\]
This proves the assertion.
\end{proof}

\begin{lem}\label{nf-leading-monomial}
Let $E$ and $E'$ be finite direct sums of $\EE_x$'s for various $x \in P$, and $\varphi:E\to E'$ a clean map. 
For any homogeneous  $\alpha\in E$ with $\varphi(\alpha) \ne 0$, 
there exists a homogeneous $\alpha'\in E$ such that $\varphi(\alpha')=\varphi(\alpha)$ and $\lv(\alpha')=\lv(\varphi(\alpha))$. 
\end{lem}

\begin{proof}
We prove the assertion by induction on $\lm(\varphi(\alpha))$ with respect to $\succ$.

\medskip
\noindent
\textbf{Case 1:}
First, we consider the case $\varphi(\alpha) \in (E')^{\nf}$. 
If $\varphi(\alpha)\in\Nf(E')$ (equivalently, $\lv(\varphi(\alpha))=0$), then we have $\deg \alpha \in \NN^n$ and  $\alpha$ can be decomposed as $\alpha=\alpha_1+\alpha_2$ for 
$\alpha_1 \in \Nf(E)$ and $\alpha_2 \in \Lf(E)$. 
Since $\varphi$ is clean, we have $\varphi(\alpha_1) \in \Nf(E')$ and $\varphi(\alpha_2) \in \Lf(E')$. Hence  $\varphi(\alpha)\in\Nf(E')$ implies $\varphi(\alpha_2)=0$ and $\varphi(\alpha)=\varphi(\alpha_1)$. So it suffices to take $\alpha'=\alpha_1$. 

So we may assume that $\varphi(\alpha) \in (E')^{\nf} \setminus \Nf(E')$. Hence we have $t_z = \lv(\varphi(\alpha))$ for some $z \in P_{\ge 2}$. By Lemma~\ref{nf-change}, we may assume that $\alpha \in E^{\nf}$. 
We have $\lm(\varphi(\delta_z^E(\alpha)))=\lm(\delta_z^{E'}(\varphi(\alpha)))\prec\lm(\varphi(\alpha))$ by Lemmas~\ref{delta-clean} and \ref{delta-lower}.
By the induction hypothesis, there exists $\beta\in E^{\nf}$ such that $\varphi(\beta)=\varphi(\delta_z^E (\alpha))=\delta_z^{E'}(\varphi(\alpha))$ and  $\lv(\beta) = \lv(\varphi(\delta_z^E(\alpha))) \preceq t_z$. 
If it can be written as $\beta=\sum_{i=1}^l r_i \cdot (u_i \otimes v_i)$, set
$$\beta\cdot t_z^{-1}:=\sum_{i=1}^l r_i \cdot (u_i\otimes v_it_z^{-1}).$$  Remark that $\delta_z^E(\beta\cdot t_z^{-1})=\beta$, $\lv(\beta\cdot t_z^{-1})=t_z$, and 
\begin{eqnarray*}
\delta_z^{E'}(\varphi(\beta\cdot t_z^{-1}))=\varphi(\delta_z^E(\beta\cdot t_z^{-1}))=\varphi(\beta)=\delta_z^{E'}\varphi(\alpha).
\end{eqnarray*}
Therefore,  $\delta_z^{E'}(\varphi(\alpha-\beta\cdot t_z^{-1}))=0$, $t_z$ does not appear in $\varphi(\alpha-\beta\cdot t_z^{-1})$. 
 If $\varphi(\alpha-\beta\cdot t_z^{-1})=0$, equivalently 
$\varphi(\alpha)=\varphi(\beta\cdot t_z^{-1})$, then $\beta\cdot t_z^{-1}$ is  a desired element.   
So we may assume that $\varphi(\alpha-\beta\cdot t_z^{-1}) \ne 0$.  
In this case, $\alpha-\beta\cdot t_z^{-1} \in E^{\nf}$, 
$\lv(\varphi(\alpha-\beta\cdot t_z^{-1})) \prec t_z$ and $\lm(\varphi(\alpha-\beta\cdot t_z^{-1})) \prec \lm(\varphi(\alpha))$. By induction hypothesis, there is $\gamma \in E^{\nf}$ such that $\varphi(\alpha-\beta\cdot t_z^{-1})=\varphi(\gamma)$ and $\lv(\gamma) = \lv(\varphi(\alpha-\beta\cdot t_z^{-1})) \prec t_z$. Hence  $\beta\cdot t_z^{-1}+\gamma$ is a desired element. 

\medskip
\noindent
\textbf{Case 2:} Next, we consider the case $\varphi(\alpha) \not \in (E')^{\nf}$. 
Then $t_m = \lv(\varphi(\alpha))$ for some $1\leq m\leq n$. Note that $t_m\varphi(\alpha) \ne 0$. 
Since $\lm(t_m\varphi(\alpha))\prec\lm(\varphi(\alpha))$, there exists $\beta\in E$ such that $\varphi(t_m\alpha)=\varphi(\beta)$ and $\lv(\beta)\preceq\lv(\varphi(t_m\alpha)) \preceq t_m$ by the induction hypothesis. 
For a monomial $u\otimes v\in \EE_x$, set
\begin{eqnarray}\label{t_m^-1 for a monomial}
(u\otimes v )\cdot t_m^{-1}:=
\begin{cases}
     (t_m^{-1}u)\otimes v  \; \; \; (\text{if} \;   \; y_m\leq x)\\
     u\otimes (t_m^{-1}v) \; \; \; (\text{if}  \; \; y_m\nleq x)\\
\end{cases}
\end{eqnarray}
and set
\begin{equation}\label{t_m^-1} 
\beta\cdot t_m^{-1}:=\sum_{i=1}^l r_i \cdot (u_i\otimes v_i)\cdot t_m^{-1}
\end{equation}
for $\beta=\sum_{i=1}^l r_i \cdot (u_i \otimes v_i)$. Remark that $t_m(\beta\cdot t_m^{-1})=\beta$ and $\lv(\beta\cdot t_m^{-1})\preceq t_m$. 
If $\varphi(\alpha-\beta\cdot t_m^{-1})=0$, then $\beta\cdot t_m^{-1} \in E$  satisfies the  desired property. So we may assume that $\varphi(\alpha-\beta\cdot t_m^{-1}) \ne 0$, 
but we can find a desired element by an argument similar to that in the last step of Case 1. 
\end{proof}

\begin{lem}\label{delta-straight}
Let $M$ be a straight module with the presentation \eqref{straight presentation}. Let $z\in P_{\geq2}$, and let
$\alpha\in M$ be a homogeneous element such that $d_M(\alpha)\in E_M^{\nf}.$
Then there is a unique element $\delta_z^M(\alpha)\in M$
such that $$d_M(\delta_z^M(\alpha))=\delta_z^{E_M}(d_M(\alpha)).$$
Moreover, if $\varphi:M\longrightarrow N$ is a clean map of straight modules, then $$\delta_z^N(\varphi(\alpha))=\varphi(\delta_z^M(\alpha)).$$
\end{lem}

\begin{proof}
Since $d'_M$ is clean, Lemma~\ref{delta-clean} gives
$$d'_M\bigl(\delta_z^{E_M}(d_M(\alpha))\bigr)=
 \delta_z^{E'_M}\bigl(d'_Md_M(\alpha)\bigr)=0.$$
By the exactness of the presentation \eqref{straight presentation}, we have 
$\delta_z^{E_M}(d_M(\alpha))\in\Im d_M.$ 
Since $d_M$ is injective, $\delta_z^M(\alpha)$ exists uniquely.

For the second assertion, let $\varphi^0:E_M\longrightarrow E_N$
be the clean map in \eqref{hom in Str}. Then
 $$\begin{aligned}
 d_N(\delta_z^N(\varphi(\alpha)))=
 \delta_z^{E_N}(d_N(\varphi(\alpha)))=\delta_z^{E_N}(\varphi^0(d_M(\alpha)))
&=\varphi^0(\delta_z^{E_M}(d_M(\alpha)))\\
&=\varphi^0(d_M(\delta_z^M(\alpha)))\\
&=d_N\varphi(\delta_z^M(\alpha)).
\end{aligned}$$
The injectivity of $d_N$ gives the assertion.
\end{proof}

\begin{prop}\label{surj-N}
Let $\varphi:M\longrightarrow N$ be a clean map. Assume that
$\Nf(\varphi)$ is surjective. Then $\varphi$ is surjective as a
$\ZZ^n$-graded $S$-homomorphism.
Moreover, for every homogeneous $0 \ne \alpha\in N$, there exists a
homogeneous $\beta\in M$ such that $\varphi(\beta)=\alpha$
and
\begin{equation}\label{LV-surj}
 \lv(d_M(\beta))
 \preceq
 \lv(d_N(\alpha)).
\end{equation}
\end{prop}

\begin{proof}
Clearly, we may assume that $\alpha \ne 0$.
We prove the  assertion  by
induction on $\lm(d_N(\alpha))$ using the notation in \eqref{hom in Str}. If $\alpha\in\Nf(N)$, the assertion follows from the surjectivity of
$\Nf(\varphi)$. Indeed, there is $\beta\in\Nf(M)$ such that
$\varphi(\beta)=\alpha$, and
$\lv(d_M(\beta))=0=\lv(d_N(\alpha)).$

So we may assume that $\alpha\notin\Nf(N)$.

\medskip
\noindent
\textbf{Case 1:} We consider the case $\lv(d_N(\alpha))=t_z$ for some $z\in P_{\geq2}$.
Since $d_N(\alpha)\in E_N^{\nf}$ now,
Lemma~\ref{delta-straight} gives $\delta_z^N(\alpha)\in N$.
By Lemmas~\ref{delta-lower} and \ref{delta-straight},
\[
 \lm(d_N(\delta_z^N(\alpha)))
 =
 \lm(\delta_z^{E_N}(d_N(\alpha)))
 \prec
 \lm(d_N(\alpha)).
\]
Hence the induction hypothesis gives a homogeneous $\beta\in M$ such
that $\varphi(\beta)=\delta_z^N(\alpha)$
and $\lv(d_M(\beta)) \preceq\lv(d_N(\delta_z^N(\alpha))) \preceq t_z.$
Set
$$d_M(\beta)\cdot t_z^{-1}:=\sum r_i\cdot (u_i\otimes v_i t_z^{-1})$$
for $d_M(\beta)=\sum r_i\cdot (u_i\otimes v_i)$.
Remark that $\delta_z^{E_M}(d_M(\beta)\cdot t_z^{-1})=d_M(\beta)$ by Lemma \ref{delta-lower} and $\lv(d_M(\beta)\cdot t_z^{-1})=t_z$.

\textbf{Subcase 1.1:} Assume that $d'_M(d_M(\beta)\cdot t_z^{-1})=0$. Then there exists $\gamma\in M$ with $d_M(\gamma)=d_M(\beta)\cdot t_z^{-1}$.
Using Lemma~\ref{delta-clean}, we obtain
\[
 \begin{aligned}
 \delta_z^{E_N}(d_N(\alpha-\varphi(\gamma)))
 &=
 \delta_z^{E_N}
 (d_N(\alpha)-\varphi^0(d_M(\gamma)))\\
 &=
  \delta_z^{E_N}
 (d_N(\alpha)-\varphi^0(d_M(\beta)\cdot t_z^{-1}))\\
  &=
 d_N(\delta_z^N(\alpha))-\varphi^0(\delta_z^{E_M}(d_M(\beta)\cdot t_z^{-1}))\\
 &=
 d_N(\delta_z^N(\alpha))-\varphi^0(d_M(\beta))\\
 &=
 d_N(\delta_z^N(\alpha)-\varphi(\beta))\\
 &=d_N(0)\\
 &=0.
 \end{aligned}
\]
Therefore, $t_z$ does not appear in $d_N(\alpha-\varphi(\gamma))$. 
If $t_x$ is a variable with $t_x \succ t_z$, then it does not appear in $\varphi^0(d_M(\gamma))=d_N(\varphi(\gamma))$ with negative powers by the cleanness of $\varphi^0$ and Lemma \ref{clean and lm}.
Furthermore, $d_N(\alpha)$ has no negative
variables larger than $t_z$. Hence $\lv(d_N(\alpha-\varphi(\gamma)))\prec t_z$.
If $\alpha-\varphi(\gamma)=0$, $\gamma \in M$ is a  desired element (since $d_M(\gamma)=d_M(\beta)\cdot t_z^{-1}$, we have $\lv(d_M(\gamma))=t_z$).
If $\alpha-\varphi(\gamma)\neq0$, then
\[
 \lm(d_N(\alpha-\varphi(\gamma)))
 \prec
 \lm(d_N(\alpha)).
\]
By the induction hypothesis, there exists $\gamma'\in M$ such that $\alpha-\varphi(\gamma)=\varphi(\gamma')$
and
$\lv(d_M(\gamma'))
 \preceq
 \lv(d_N(\alpha-\varphi(\gamma)))
 \prec t_z.
$
Thus $\alpha=\varphi(\gamma+\gamma')\in \Im\varphi$. Since $\lv(d_N(\alpha))=t_z=\lv(d_M(\gamma))=\lv(d_M(\gamma+\gamma'))$, $\gamma+\gamma' \in M$ is a desired element.

\textbf{Subcase 1.2:} Assume that $d'_M(d_M(\beta)\cdot t_z^{-1})\neq 0$. Note that $t_z\succeq \lv(d_M'(d_M(\beta)\cdot t_z^{-1}))$ by the cleanness of $d'_M$ and Lemma~\ref{clean and lm}. By Lemma \ref{nf-leading-monomial}, there exists $\gamma\in E_M$ such that 
\begin{equation}\label{beta and gamma_2}
d'_M(d_M(\beta)\cdot t_z^{-1})=d'_M(\gamma)
\end{equation}
and $\lv(\gamma)=\lv(d'_M(d_M(\beta)\cdot t_z^{-1}))$.
Since $$\delta_z^{E_M'}(d'_M(d_M(\beta)\cdot t_z^{-1}))=d'_M(d_M(\beta))=0,$$ $t_z$ does not appear in $d'_M(d_M(\beta)\cdot t_z^{-1})$. 
Hence 
\begin{equation}\label{lv(gamma)}
\lv(\gamma)=\lv(d'_M(d_M(\beta)\cdot t_z^{-1})) \prec t_z.
\end{equation}
If $\varphi^0(\gamma) \ne 0$, then $\lv(\varphi^0(\gamma)) \prec t_z$ by Lemma~\ref{clean and lm}. 
On the other hand, 
\begin{align*}
\delta_z^{E_N}(d_N(\alpha)-\varphi^0(d_M(\beta)\cdot t_z^{-1}))&= d_N(\delta_z^N(\alpha))-\varphi^0(d_M(\beta))\\
&= d_N(\varphi(\beta))-d_N(\varphi(\beta))\\
&=0,  
\end{align*}
and $t_z \not \in \nv(d_N(\alpha)-\varphi^0(d_M(\beta)\cdot t_z^{-1}))$. 
If $t_x \succ t_z$, then $t_x \not \in \nv(d_N(\alpha)-\varphi^0(d_M(\beta)\cdot t_z^{-1}))$ by Lemma~\ref{clean and lm}. Hence $\lv(d_N(\alpha)-\varphi^0(d_M(\beta)\cdot t_z^{-1})) \prec t_z$. 

Note that $d'_M(d_M(\beta)\cdot t_z^{-1}-\gamma)=0$ by \eqref{beta and gamma_2}, and we take $\gamma'\in M$ such that $d_M(\gamma')=d_M(\beta)\cdot t_z^{-1}-\gamma$. 
Remark that 
\begin{equation}\label{dM(delta)}
\lv(d_M(\gamma'))=\lv(d_M(\beta)\cdot t_z^{-1})=t_z 
\end{equation}
by \eqref{lv(gamma)}. If $\alpha=\varphi(\gamma')$, then $\gamma' \in M$ is a desired element. So we may assume that $\alpha \ne  \varphi(\gamma')$. Then 
\begin{eqnarray*}
d_N(\alpha-\varphi(\gamma'))&=&d_N(\alpha)-d_N(\varphi(\gamma'))\\
&=& d_N(\alpha)-\varphi^0(d_M(\gamma'))\\
&=&d_N(\alpha)-\varphi^0(d_M(\beta)\cdot t_z^{-1}-\gamma)\\
&=&d_N(\alpha)-\varphi^0(d_M(\beta)\cdot t_z^{-1})+\varphi^0(\gamma). 
\end{eqnarray*}
We have $\lv(d_N(\alpha-\varphi(\gamma')))\prec t_z=\lv(d_N(\alpha))$, since a variable $t_x$ with $t_x \succeq t_z$ does not appear in $d_N(\alpha)-\varphi^0(d_M(\beta)\cdot t_z^{-1})$ and $\varphi^0(\gamma)$ with negative powers. 
By the induction hypothesis, $\alpha-\varphi(\gamma')=\varphi(\gamma'')$ and $\lv(d_M(\gamma''))\preceq\lv(d_N(\alpha-\varphi(\gamma'))) \, (\prec t_z)$ hold for some $\gamma''\in M$. 
Therefore, $\alpha=\varphi(\gamma'+\gamma'') \in \Im \varphi$,  
and $t_z=\lv(d_N(\alpha))=\lv(d_M(\gamma'+\gamma''))$ 
by \eqref{dM(delta)}.
So we are done. 

\textbf{Case 2:}
Next we consider the case where $t_m =\lv(d_N(\alpha))$ for some $1 \le m \le n$. 
Since $d_N(t_m\alpha) = t_m \cdot d_N(\alpha)\ne 0$ and 
$\lm(d_N(\alpha)) \succ\lm(d_N(t_m\alpha))$, there exists $\beta\in M$ with $t_m\alpha=\varphi(\beta)$ and $$(t_m \succeq) \, \lv(d_N(t_m\alpha)) \succeq \lv(d_M(\beta))$$ by the induction hypothesis. 
In the same way as \eqref{t_m^-1 for a monomial} and 
\eqref{t_m^-1}, we take 
$d_M(\beta)\cdot t_m^{-1} \in E_M$. Note that $t_m(d_M(\beta)\cdot t_m^{-1})=d_M(\beta)$, and $t_m\succeq \lv(d_M(\beta)\cdot t_m^{-1})$. 

\textbf{Subcase 2.1:}
Assume that $d'_M(d_M(\beta)\cdot t_m^{-1})=0$. 
This subcase follows similarly to Subcase 1.1, with the difference that some of the equalities are replaced by $\succeq$. 
We take $\delta\in M$ such that $d_M(\delta)=d_M(\beta)\cdot t_m^{-1}$. Note that 
$\lv(d_M(\delta))\preceq t_m$. 
By the same argument as in Subcase 1.1, we can show that 
$t_m\cdot d_N(\alpha-\varphi(\delta))=0$.  
Therefore, $t_m \not \in \nv(d_N(\alpha-\varphi(\delta)))$. If $\alpha=\varphi(\delta)$, then $\delta \in M$ is a  desired element. So we may assume that $\alpha\ne\varphi(\delta)$. 
For $l <m$, $t_l$ does not appear in $d_N(\varphi(\delta))=\varphi^0(d_M(\delta))=\varphi^0(d_M(\beta)\cdot t_m^{-1})$  with negative powers.  
So we have $\lm(d_N(\alpha-\varphi(\delta)))\prec \lm(d_N(\alpha))$, and there exists $\delta'\in M$ such that $\alpha-\varphi(\delta)=\varphi(\delta')$ and $t_m \succ\, \lv(d_N(\alpha-\varphi(\delta)))\succeq\lv(d_M(\delta'))$ by the induction hypothesis. Since $d_M(\delta)=d_M(\beta)\cdot t_m^{-1}$, $\alpha=\varphi(\delta+\delta')\in \Im\varphi$ with $t_m=\lv(d_N(\alpha))\succeq \lv(d_M(\delta+\delta'))$. So we are done.

\textbf{Subcase 2.2:}
Assume that $d'_M(d_M(\beta)\cdot t_m^{-1})\neq 0$. 
 The proof of this subcase is analogous to that of Subcase 1.2,  while some of the equalities are replaced by $\succeq$, and $t_m$ is multiplied instead of applying the operations $\delta_z^*$, as in Subcase 2.1.  
\end{proof}

\begin{cor}\label{surj-epi}
For $\varphi\in \Hom_{\Str_P S}(M,N)$, the following
conditions are equivalent.
\begin{enumerate}
\item $\varphi$ is an epimorphism in $\Str_P S$.
\item $\varphi$ is surjective as a homomorphism in $\*Mod S$.
\end{enumerate}
\end{cor}

\begin{proof}
Assume first that $\varphi$ is an epimorphism in $\Str_P S$.
Since $\Nf:\Str_P S\longrightarrow\Sq A_P$ is an equivalence, $\Nf(\varphi)$ is an epimorphism in $\Sq A_P$.
Hence $\Nf(\varphi):\Nf(M)\longrightarrow\Nf(N)$
is surjective in $\*Mod S$. Proposition~\ref{surj-N} implies that $\varphi$ is
surjective in ${}^*\!\Mod S$.

Conversely, suppose that $\varphi$ is surjective in $\*Mod S$. If
$f,g:N\longrightarrow L$ are morphisms in $\Str_P S$ satisfying
$f\circ\varphi=g\circ\varphi$, then the equality holds as ordinary $S$-module homomorphisms. Since $\varphi$
is surjective, $f=g$. Thus $\varphi$ is an epimorphism in $\Str_P S$.
\end{proof}

\begin{thm}\label{Coker is straight}
For a morphism $\varphi: M \to N$ in $\Str_P S$, we have the following. 
\begin{itemize}
    \item[(1)] Let $\Im\varphi$ (resp. $I_\varphi$) be the image of $\varphi$ in $\*Mod S$ (resp. $\Str_P S$). Then $\Im \varphi = I_\varphi$ as submodules of $N$. In particular, $\Im\varphi$ is straight.
    \item[(2)] Let $C_\varphi$ be the cokernel of $\varphi$ in $\Str_P S$. 
    Then $\Coker \varphi \cong C_\varphi$ in $\*Mod S$. In particular, $\Coker\varphi$ is straight.
\end{itemize}
\end{thm}

\begin{proof}
(1)  Since $\Nf:\Str_P S \to \Sq A_P$ gives a categorical equivalence 
by Theorem \ref{sq-str}, $\Nf(I_\varphi)$ coincides with the image of $N(\varphi): \Nf(M) \to \Nf(N)$. 
Hence there exist homomorphisms $\varphi'\in [\! \gHom_S(\Nf(M), \Nf(I_\varphi))]_\b0$ and  $\iota\in [\! \gHom_S(\Nf(I_\varphi), \Nf(N))]_\b0$ such that $\varphi'$ is surjective, $\iota$ is injective and $\Nf(\varphi)=\iota\cdot \varphi'$. Taking the \v{C}ech hulls, we obtain the graded homomorphisms $\C(\varphi')\in  [\gHom_S(M,I_\varphi)]_\b0$ and $\C(\iota)\in [\gHom_S(I_\varphi, N)]_\b0$ with  $\varphi=\C(\iota)\cdot \C(\varphi')$. Since $\C$ gives a categorical equivalence, $\C(\iota)$ is a monomorphism and $\C(\varphi')$ is an epimorphism in $\Str_P S$. By Corollaries~\ref{inj-mono} and \ref{surj-epi}, they are injective and surjective, respectively, as graded $S$-homomorphisms.

By the universality of $\Im\varphi$ in $\*Mod S$, there exists $\tau\in [\gHom_S(\Im\varphi, I_\varphi)]_\b0$ such that the diagram
\begin{center}
\begin{tikzpicture}
\node (M) at (0,0) {$M$};
\node (N) at (4,0) {$N$};
\node (I) at (2,-1.2) {$\Im\varphi$};
\node (Iv) at (2,-3) {$I_\varphi$};

\draw[->] (M) -- node[above] {$\varphi$} (N);

\draw[->] (M) -- node[right]{}  (I);

\draw[->, thick] (I) -- node[left] {} (N);

\draw[->] (M) -- node[below left] {$\C(\varphi')$} (Iv);

\draw[->] (Iv) -- node[below right] {$\C(\iota)$} (N);

\draw[->, dashed] (I) -- node[right] {$\tau$} (Iv);

\end{tikzpicture}
\end{center}
is commutative, where $\tau$ is injective. On the other hand, $\tau$ is surjective by the following diagram 
\[
  \begin{CD}
     M  @>>>  \Im\varphi @>>>  0  \\
      @VV{ \id_M}V  @VV{\tau}V  @.  \\
    M @>{\C(\varphi')}>>  I_\varphi @>>>  0.
  \end{CD}
\]
Therefore, $\tau$ is an isomorphism.

(2)  Let $(I_\varphi, \iota)$ and 
$(K_\pi, \iota')$ be the image of $\varphi$ and the kernel of $\pi$ in $\Str_P S$, respectively. 
Note that 
$$0 \too I_\varphi \stackrel{\iota}{\too}  N \stackrel{\pi}{\too}  C_\varphi \too  0$$
is exact in $\Str_P S$.  We see that it is also exact in $\*Mod S$. In fact, $\iota$ is injective and $\pi$ is surjective also in $\*Mod S$ by Corollaries~\ref{inj-mono} and \ref{surj-epi}. And the exactness at $N$ follows from Theorem~\ref{sq-str} and (1) of this theorem. 
By the universality of $\Im\varphi$ and $\Coker\varphi$,
we have the following diagram
\[
  \begin{CD}
   0 @>>>\Im \varphi @>>>  N  @>>>  \Coker \varphi @>>>  0  \\
     @.  @VV{\tau}V  @VV{ \id_N}V  @VV{\tau'}V  @.  \\
    0 @>>>I_\varphi @>>> N @>>>  C_\varphi @>>>  0,
  \end{CD}
\]
where $\tau$ is the isomorphism in (1). 
By the five lemma, $\tau'$ is also.  
\end{proof}

\section{Local cohomology modules as straight modules}
Recall that $\omega_S$ denotes the $\ZZ^n$-graded canonical module of $S$, and $\fm=\fp_{\hat{0}} =(t_x \mid x\in P^*)\subset S$ is the graded maximal ideal. 
Since $S_x =K[t_i \mid y_i \le x]$, we have
\begin{equation}\label{local cohomologies of S_x}
H_\fm^i(S_x) =\begin{cases}
K[t_i^{-1} \mid y_i \le x] \, (\deg t_x) & (\text{if $i=\dim  S_x$}), \\
0 & (\text{otherwise}). 
\end{cases}    
\end{equation}

For $M \in \Sq A_P$,  there exists a minimal injective resolution
\begin{equation}\label{injeresol in Sq}
I^\bullet_M: 0\too M \too I^0_M \stackrel{\partial^0}{\too} I_M^1 \stackrel{\partial^1}{\too} \cdots \stackrel{\partial^{t-1}}{\too} I_M^t\too 0,
\end{equation}
$$I_M^i=\bigoplus_{x\in P} S_x^{\nu^i(x, M)}$$
in $\Sq A_P$. Set $\nu^i(\fm, M):= \nu^i(\hat{0}, M)$ for any $i$, and $\nu(x, M):=\nu^0(x, M)$.

\begin{lem}\label{inj-lem}
Let $M$ be a squarefree module, and $I_M$ its injective envelope in $\Sq A_P$.
\begin{itemize}
\item[(1)] $\dim_K [H_\fm^i(I_M)]_{\mathbf{0}}=\begin{cases}
\nu(\fm, M) & (\text{if $i=0$}),\\
0 & (\text{otherwise}). 
\end{cases}$
\item[(2)] $\dim_K [H_\fm^0(M)]_{\mathbf{0}}=\nu(\fm, M)$.
\item[(3)] For the exact sequence $0\to M\to I_M \to N \to 0$, we have
\begin{eqnarray*}
[H_\fm^i(N)]_{\mathbf{0}}\cong 
[H_\fm^{i+1}(M)]_{\mathbf{0}} \; \; \; \mathrm{if}\; \; i\geq 0
\end{eqnarray*}
\end{itemize}
\end{lem}
\begin{proof}
(1) With the notation of \eqref{injeresol in Sq}, we have 
$I_M=I_M^0=\bigoplus_{x\in P} S_x^{\nu(x, M)}$.
Now, the assertion follows from \eqref{local cohomologies of S_x}. 

(2) We have  $M\cong M' \oplus K^m$ for some $m \ge 0$ and  $M' \in \Sq A_P$ with $\Gamma_\fm(M')=0$. Since $S_{\hat{0}}=K$ and $I_M^0$ is the injective envelope of $M$ in $\Sq A_P$, we have $\nu(\fm, M)=m$ and $I_M^0=I_{M'}^0 \oplus K^m$.   
Since $H_\fm^0(M)=H^0(\Gamma_\fm(I_M^\bullet))$ and $\partial^0(K^m)=0$, we are done. 

(3) Consider the long exact sequence
$$
\cdots \to  H_\fm^{i-1}(N)\to H_\fm^i(M) \to H_\fm^i(I_M) \to H_\fm^i(N) 
\to H_\fm^{i+1}(M)\to H_\fm^{i+1}(I_M).$$
If $i>0$, then $[H_\fm^i(I_M)]_\mathbf{0}=0$ by (1), hence $[H_\fm^i(N)]_\mathbf{0}=[H_\fm^{i+1}(M)]_\mathbf{0}$. By (1) and (2), $[H_\fm^0(M)]_\mathbf{0} \to [H_\fm^0(I_M)]_\mathbf{0}$ is an isomorphism, and $[H_\fm^0(N)]_\mathbf{0}=[H_\fm^1(M)]_\mathbf{0}$.
\end{proof}

Recall that $p=\dim S= \#P^*$. 

\begin{prop}\label{bass-sq}
Let $M$ be a squarefree module. Then, we have
$$\nu^i(\fm, M)=\dim_K[H_\fm^i(M)]_\mathbf{0}=\dim_K [\gExt_S^{p-i}(M,\omega_S)]_\mathbf{0}.$$
\end{prop}
\begin{proof}
The second equality follows from the $\ZZ^n$-graded local duality 
$$H_\fm^i(M)^\vee=\gExt_S^{p-i}(M,\omega_S),$$
where $(-)^\vee$ means the $\ZZ^n$-graded Matlis dual. 
Here we use a somewhat unusual $\ZZ^n$-grading, but the local duality theorem still holds (see \cite[14.4.1]{BS}). So it suffices to prove the first equality. If $i=0$, the first equality has been proved in Lemma \ref{inj-lem} (2). So we may assume that $i>0$. We use induction on $\injdim_{\Sq A_P} M$. If $\injdim_{\Sq A_P}M=0$, the assertion is nothing other than Lemma \ref{inj-lem} (1), since $M\cong I_M$ in this case. Assume that $\injdim_{\Sq A_P} M>0$. Consider $0\to M\to I_M\to N\to 0$. 
Note that $N$ is squarefree again, and $\nu^i(\fm, M)=\nu^{i-1}(\fm, N)$ for all $i\geq 1$. 
Since $\injdim_{\Sq A_P} N=\injdim_{\Sq A_P} M-1$, we have $\nu^{i-1}(\fm, N)=\dim_K[H_\fm^{i-1}(N)]_\mathbf{0}$ by the induction hypothesis. By Lemma \ref{inj-lem} (3), we have $\nu^i(\fm, M)=\dim_K [H_\fm^{i-1}(N)]_\mathbf{0}=\dim_K[H_\fm^i(M)]_\mathbf{0}$ for all $i\geq 1$. 
\end{proof}

\begin{prop}\label{bass-str}
Let $M$ be a straight module and $\EE^\bullet : 0\to M\to \EE^0\to {}^*\! E^1 \to \cdots$ a minimal injective resolution of $M$ in $\*Mod S$.
\begin{itemize}
\item[(1)] The category $\Str_P S$ has enough injectives, and every injective object is isomorphic to $\bigoplus_{x\in P} \EE_x^{m_x}$ for some $m_x\in \NN$. 

\item[(2)] The complex $\EE^\bullet$ is isomorphic to a minimal injective resolution of $M$ in $\Str_P S$. Thus, $\EE^i\cong \bigoplus_{x\in P}  \EE_x^{\mu^i(\fp_x, M)}$, $\mu^i(\fp_x,M)<\infty$, as $\ZZ^n$--graded $S$--modules (the degree shifting does not occur).
\item[(3)] $\Nf({}^*\! E^\bullet)$ is a minimal injective resolution of $\Nf(M)\in \Sq A_P$ in $\Sq A_P$.  Thus $\mu^i(\fp_x, M)=\nu^i(x, \Nf(M))$ for all $x\in P$ and $i\geq 0$.
\end{itemize}
\end{prop}

\begin{proof}
(1): Since $\Str_P S\cong \Sq A_P$ and $\Nf(\EE_x)$ is isomorphic to $S_x$, which is injective in $\Sq A_P$, the assertion follows.

(2)  Let $I^\bullet$ be a minimal injective resolution of $\Nf(M)$ in $\Sq A_P$. 
Then the \v{C}ech hull $\C(I^\bullet)$ is a minimal injective resolution of $M$ in $\*Mod S$. In fact, $\C(S_x)$ is isomorphic to $\EE_x$, which is injective in $\*Mod S$. Hence $\C(I^\bullet)$ is an injective resolution of $M$ in $\*Mod S$. If $\check C(I^\bullet)$ had a nonzero trivial direct summand, applying $\Nf$ would yield a nonzero trivial direct summand of $I^\bullet$, contradicting its minimality.
So the assertion follows. 

(3): The first assertion follows from (2), and the second follows from the first. 
\end{proof}

The next theorem generalizes \cite{Mu00, T99}.

\begin{thm}\label{localcoh-ext}
The local cohomology module $H_{I_P}^i(\omega_S) \, (\cong H_{I_P}^i(S)(-\bw))$ is a straight module, and $\Nf (H_{I_P}^i(\omega_S))$ is isomorphic to $\gExt^i_S(A_P,\omega_S)$.
\end{thm}

\begin{proof}
Let $\Gamma_{I_P}: \*Mod S \to \*Mod S$ be the $I_P$-torsion functor. So it gives  
$$
\Gamma_{I_P}(\EE(S/\fp))=
\begin{cases}
\EE(S/\fp) & \text{$\fp \supset I_P$, equivalently, $\fp=\fp_x$ for $\exists x \in P$},\\
0 & \text{otherwise} 
\end{cases}
$$
for a $\ZZ^n$-graded prime ideal $\fp$. 

Let $\DB_S$ be the $\ZZ^n$-graded dualizing complex of $S$. 
Then $\IB:={}^*\! D_S^{\bullet -p}$ is a minimal graded injective  resolution of $\omega_S$. 
The $i$-th cohomology $H^i(\Gamma^\bullet)$ of the cochain complex $\Gamma^\bullet :=\Gamma_{I_P}(\IB)$  
is the $i$-th local cohomology $H_{I_P}^i(\omega_S)$.  
As we have shown in \cite[\S 4]{SY}, $\Gamma^\bullet$ is of the form 
$$0 \too \Gamma^{p-d}\too \Gamma^{p-d+1}\too \cdots \too \Gamma^{p-1}\too \Gamma^p\too 0$$ 
with 
$$\Gamma^{i}=\bigoplus_{\substack{x \in P \\\rank x=p-i}} \EE_x.$$
Moreover, the differential maps of $\Gamma^\bullet$ are clean by \cite[Theorem~4.5]{SY}, and $\Gamma^\bullet$ is a cochain complex in $\Str_P S$.   By Theorems~\ref{Ker is straight} and \ref{Coker is straight}, $H^i_{I_P}(\omega_S)$ is straight. 

Since  the functor $\Nf: \Str_P S \to \Sq A_P$ is exact and $\Nf(\Gamma^\bullet)\cong [\Hom_S(A_P, \Gamma^\bullet)]_{\ge \b0}$, we have 
\begin{align*}
\Nf(H_{I_P}^i(\omega_S)) \cong \Nf(H^i(\Gamma^\bullet)) \cong H^i(\Nf(\Gamma^\bullet)) &\cong H^i( [\Hom_S(A_P, \Gamma^\bullet)]_{\ge \b0})\\ & \cong H^i( [\Hom_S(A_P, {}^*\!I_S^\bullet)]_{\ge \b0})\\
&\cong \gExt_S^i(A_P, \omega_S).
\end{align*}
\end{proof}

\begin{prop}\label{simple in str}
Let $\omega_{S_x}:=S_x(-\deg t_x)$ be the $\ZZ^n$-graded canonical module of $S_x$. Then the following hold. 
\begin{enumerate}
\item A straight module $M$ is simple in $\Str_P S$ (i.e. there is no non-trivial straight submodule) if and only if $M\cong \C(\omega_{S_x})$ 
for some $x\in P$.
\item For $M \in \Str_P S$,  there is a chain of straight  submodules $0=M_0\subset M_1\subset \cdots \subset M_t=M$ such that each quotient is isomorphic to $\C(\omega_{S_x})$ for some $x \in P$.
\item For $x \in P$ with $\rank x=r$, $\C(\omega_{S_x}) \cong H_{\fp_x}^{p-r}(\omega_{S})$. 
\end{enumerate}
\end{prop}

\begin{proof}
(1) and (2) follow from Lemma~\ref{simple in sq} and Theorem~\ref{sq-str}. 

(3) Recall  that $\gExt_S^{p-r}(S_x, \omega_S) \cong \omega_{S_x}$, and $\gExt_S^i(S_x, \omega_S)=0$ for $i \ne p-r$. 
Let $I_S^\bullet$ be the minimal injective resolution of $\omega_S$ in $\*Mod S$.  As shown in the proof of Lemma~\ref{Ext is sq}, $\omega_{S_x}$ is the $(p-r)$-th cohomology of $I_x^\bullet:=[\gHom_S(S_x, I^\bullet_S)]_{\ge \b0}$, which has the form  
$$I_x^\bullet: 0 \to I_x^{p-r} \to I_x^{p-r+1} \to \cdots \to I_x^p \to 0$$
with
$$I_x^i=\bigoplus_{\substack{z \in P, z \le x \\ \rank z = p-i}} S_z.$$

Since 
$$\C(I^\bullet_x): 0 \to \C(I_x^{p-r}) \to \C(I_x^{p-r+1}) \to \cdots \to \C(I_x^p) \to 0$$
with
$$\C(I_x^i)=\bigoplus_{\substack{z \in P, z \le x \\ \rank z = p-i}} \EE_z,$$
this cochain complex is isomorphic to the complex $\Gamma_{\fp_x}(I_S^\bullet)$, whose $i$-th cohomology is $H_{\fp_x}^i(\omega_S)$. 
So we are done. 
\end{proof}


\begin{rem}
(1) We are not sure whether $\Str_P S$ is closed under extensions in $\*Mod S$. Thus, the converse of Proposition~\ref{simple in str} (2) might not hold. 

(2) If $\#P \ge 3$, equivalently, $I_P \ne 0$, then $H_{\fp_x}^{p-r}(\omega_{S})$ for 
$r:=\rank x= \dim S_x$ is not finitely generated as an $S$-module for any $x\in P$. We now verify this fact. We have $$H_{\fp_x}^{p-r}(S) \otimes_S S_{\fp_x}\cong H_{\fp_xS_{\fp_x}}^{p-r}(S_{\fp_x}).$$ Since $\fp_x \supset I_P \neq 0$,  $S_{\fp_x}$ is a regular ring of positive dimension containing a field $K$, the right side is not finitely generated. Therefore, $H_{\fp_x}^{p-r}(\omega_{S})$ is not finitely generated. Since any non-zero straight module $M$ contains $H_{\fp_x}^{p-r}(\omega_{S})$ as a submodule for some $x \in P$
by Proposition \ref{simple in str}, $M$ is not finitely generated. 

\end{rem}

The following corollary follows immediately from Theorem \ref{localcoh-ext}.

\begin{cor}
We have $$\mathrm{cd}(I_P) \, (:=\max\{i \mid H_{I_P}^i(S)\neq 0\}) =\dim S-\depth(S/I_P).$$
\end{cor}

\begin{proof}
By Theorem \ref{localcoh-ext}, we have $\mathrm{cd}(I_P)=\max\{ i \mid \gExt_S^i(A_P, \omega_S)\neq 0\}$, where the right side equals $\dim S-\depth(S/I_P)$.
\end{proof}

\begin{thm}\label{Lyubeznik number}
The Lyubeznik numbers of  $A_P$ are given by
\begin{equation}\label{double Ext}
\lambda_{i,j}(A_P)=\dim_K [\gExt_S^{p-i}(\gExt_S^{p-j}(A_P,\omega_S),\omega_S)]_{\mathbf{0}}.
\end{equation}
\end{thm}
\begin{proof}
From Propositions \ref{bass-sq}, \ref{bass-str} and Theorem \ref{localcoh-ext},
\begin{eqnarray*}
\lambda_{i,j}(A_P)&=&\mu^i(\fm, H_{I_P}^{p-j}(\omega_S)) \\
&=&\nu^i(\fm, \Nf( H_{I_P}^{p-j}(\omega_S))\\
&=&\nu^i(\fm, \gExt_S^{p-j}(A_P,\omega_S))\\
&=&\dim_K[\gExt_S^{p-i}(\gExt_S^{p-j}(A_P,\omega_S), \omega_S)]_{\mathbf{0}}.
\end{eqnarray*}
\end{proof}

\begin{rem}
Assume that $\chara(K)>0$. Then we see that  $A_P$ is {\it $F$-pure}. In fact, since the set of standard monomials forms a $K$-basis of $A_P$,  the argument for Stanley–Reisner rings in \cite[Example~2.2]{SW} carries over to $A_P$. 
If a standard graded ring  $R=\bigoplus_{i \ge 0} R_i$ with $R_0=K$ is $F$-pure, then \cite{DGN} shows that the equation corresponding to \eqref{double Ext} holds for $R$. However, we cannot use their result directly, since $A_P$ is not standard graded.    
\end{rem}

\section{Lyubeznik numbers are topological invariants}
Recall that a simplicial poset $P$ gives a regular cell complex $\Gamma(P)$. Let $X:=|\Gamma(P)|$ be the underlying space of $\Gamma(P)$, and $c(x)$ the open cell corresponding to $x\in P^*$. Hence, for each $x\in P^* \setminus \{y_1, \ldots, y_n\}$, $c(x)$ is a subset of $X$ homeomorphic to $\mathbb{R}^{\rank x-1}$, and $c(y_i)$ is a single point for all $i$. Moreover $X$ is the disjoint union of the cells $c(x)$, and $x\geq y$ if and only if $\ol{c(x)} \supset c(y)$.

As in the preceding section, let $\Lambda$ be the incidence algebra of $P$, and $\mod \Lambda$ the category of finitely generated left $\Lambda$--modules.
In \cite{Y05}, a constructible sheaf $N^\dagger$ on $X$ is associated with each $N \in \mod \Lambda$.  This construction is essentially the same as that in \cite{Sh}. 
By the equivalence $\Sq A_P\cong \mdL$, we have the constructible sheaf $M^+$ on $X$ corresponding to $M\in \Sq A_P$ (\cite{Y11}). Here, we give a precise construction for the reader's convenience. For further details on the theory of sheaves, see \cite{Iv}. 

For $x\in P$ with $r=\rank x$, set $\ua(x):=\sum_{y_i \le x} r \be_i^x \in \MM(x)\subset \MM$. For $M\in \Sq A_P$, set
$$\sp(M):=\bigcup_{x\in P^*} c(x) \times M_{\ua(x)}.$$
 Let $\pi:\sp(M)\to X$ be the projection sending $(p,m)\in c(x)\times M_{\ua(x)}\subset \sp(M)$ to $p\in c(x)\subset X$.
For an open subset $U\subset X$ and a map $s:U \to \sp(M)$, we consider the following conditions.
\begin{itemize}
\item[(C1)] $\pi\circ s=\id_U$ and $s_p=t^{\ua(x)-\ua(y)}\cdot s_q$ for all $p \in c(x)\cap U, q\in c(y)\cap U$ with $x\geq y$. Here, $s_p\in M_{\ua(x)}$ (resp. $s_q\in M_{\ua(y)}$) is the element with $s(p)=(p,s_p)$ (resp. $s(q)=(q,s_q)$).
\item[(C2)] There is an open covering $U=\bigcup_{i\in I} U_i$ such that the restriction of $s$ to $U_i$ satisfies (C1) for all $i\in I$.
\end{itemize}

We define a sheaf $M^+$ on $X$ as follows. For an open set $U\subset X$, set
$$M^+(U):=\{s\mid s:U\to \sp(M)\;  \text{is a map satisfying (C2)}\},$$
and let the restriction map $M^+(U)\to M^+(V)$ for $U\supset V$ be the natural one. Then $M^+$ is a constructible sheaf with respect to the cell decomposition $\Gamma(P)$. Note that $(A_P)^+$ is the $K$--constant sheaf $\const_X$ on $X$, and $(S_x)^+$ is the extension to $X$ of the $K$--constant sheaf on the closed cell $\overline{c(x)}$. 
An element $\alpha\in M_\b0$ naturally determines a global section $s_\alpha \in M^+(X)$. Namely, for $p\in X$ with $p\in c(x)$, we define the value of $s_\alpha$ at $p$ to be $t^{\ua(x)}\alpha \in M_{\ua(x)}$. 
Let $\Sh(X)$ be the category of sheaves of $K$--vector spaces on $X$ with finite-dimensional stalks. We have the functor $(-)^+: \Sq A_P\to \Sh(X)$, which is exact.

\begin{thm}[cf. \cite{Y11}, Theorem 4.2]\label{local-sheaf coh}
In the above setting, we have
$$H^i(X; M^+)\cong [H_\fm^{i+1}(M)]_\b0$$ 
for all $i\geq 1$, and an exact sequence
\begin{equation}\label{local-seq}
0\to [H_\fm^0(M)]_\b0 \to M_0 \to H^0(X; M^+)\to [H_\fm^1(M)]_\b0 \to 0.
\end{equation}
\end{thm}

Note that $X$ admits Verdier's dualizing complex $\cDx$ with coefficients in $K$, and recall that $M \in \Sq A_P$ implies 
$\gExt_S(M, \omega_S) \in \Sq A_P$ by Lemma~\ref{Ext is sq}.

\begin{thm}[cf. \cite{AY}, Theorem 5.5]\label{Ext-correspondence}
With the above notation, we have the following:
\begin{enumerate}
\item[(1)] We have 
 $$\gExt_S^{p-i}(M,\omega_S)^+ \cong \SExt^{1-i}(M^+, \cDx)$$
 for all $0 \le i \le p$ and 
 $$[\gExt_S^{p-i}(M,\omega_S)]_\b0 \cong \Ext^{1-i}(M^+, \cDx)$$
 for $2 \le i \le p$. 
\item[(2)] Under the second isomorphisms in (1), for $i\geq 2$, the natural map 
$$\Ext^{1-i}(M^+, \cDx)\to \Gamma(X; \SExt^{1-i}(M^+, \cDx))$$
coincides with the middle map
$$[\gExt_S^{p-i}(M,\omega_S)]_\b0 \to \Gamma(X; \gExt_S^{p-i}(M,\omega_S)^+)$$
of (\ref{local-seq}).
\end{enumerate}
\end{thm}

\begin{proof}
By appropriately modifying \cite[Theorem 4.2]{Y03} and \cite[Lemma 5.11]{Y08} for the present case, and further using \cite[Theorem 4.1]{Y11}, we obtain the desired result.
\end{proof}

\begin{thm}\label{topological}
The Lyubeznik numbers $\lambda_{i,j}(A_P)$ depend only on the homeomorphism class of $X=|\Gamma(P)|$ and $\chara(K)$.
\end{thm}

\begin{proof}
Following the argument used in the proof of Theorem 5.3 in \cite{AY}, the theorem follows directly from the theorems in this section and Theorem \ref{Lyubeznik number}. We note that, as in \cite{AY}, the cases $j=0,1$ need to be treated separately.
\end{proof}

\section*{Acknowledgments}
We appreciate the early discussions with Professor Núñez-Betancourt. We are deeply grateful to Professors Mitsuyasu Hashimoto and Mitsuhiro Miyazaki for their valuable comments. ChatGPT (OpenAI) was used only for language editing and supplementary checks of the manuscript, including its mathematical arguments. The authors take full responsibility for the content of the paper.

The first author is supported by JSPS KAKENHI Grant Number 22K03258. The second author is supported by JSPS KAKENHI Grant Number 25K06928.

\end{document}